\documentclass[11pt, centertags,oneside]{amsart}

\usepackage{amsmath,amstext,amsthm,amscd,typearea,hyperref,stmaryrd}
\usepackage{amssymb}
\usepackage{a4wide}
\usepackage[mathscr]{eucal}
\usepackage{mathrsfs}
\usepackage{typearea}
\usepackage{charter}
\usepackage{pdfsync}
\usepackage{MnSymbol}
\usepackage{framed}
\usepackage[a4paper,width=16.2cm,top=3cm,bottom=3cm]{geometry}

\numberwithin{equation}{section}

\usepackage{xcolor}
\usepackage{tikz}

\newtheorem{theorem}{Theorem}[section]
\newtheorem{definition}[theorem]{Definition}
\newtheorem{proposition}[theorem]{Proposition}
\newtheorem{corollary}[theorem]{Corollary}
\newtheorem{lemma}[theorem]{Lemma}
\newtheorem{remark}[theorem]{Remark}

\newtheorem{question}[theorem]{Question}

\newcommand{\cali}[1]{\mathscr{#1}}

\newcommand{\supp}{{\rm supp}}

\newcommand{\dist}{{\rm dist}}

\newcommand{\id}{{\rm id}}

\newcommand{\reg}{{\rm reg}}

\newcommand{\codim}{{\rm codim\ \!}}

\newcommand{\Cc}{\cali{C}}
\newcommand{\Dc}{\cali{D}}

\newcommand{\Lc}{\cali{L}}

\newcommand{\Qc}{\cali{Q}}

\newcommand{\Uc}{\cali{U}}

\newcommand{\C}{\mathbb{C}}

\newcommand{\N}{\mathbb{N}}

\newcommand{\R}{\mathbb{R}}

\renewcommand{\P}{\mathbb{P}}

\newcommand{\dsh}{{\mathrm{DSH}}}

\newcommand{\Xf}{{\mathfrak{X}}}
\newcommand{\Xfh}{{\widehat{\mathfrak{X}}}}

\newcommand{\dyn}{\mathrm{dyn}}

\newcommand{\pmain}{p_\mathrm{main}}

\title[]{Inverse images of positive closed currents\\ under holomorphic endomorphisms of compact K\"ahler manifolds}

\author{Taeyong Ahn}
\address{(Ahn) Department of Mathematics Education, Inha University, 100 Inha-ro, Michuhol-gu, Incheon 22212, Republic of Korea}%
\email{t.ahn@inha.ac.kr}

\date{\today}

\keywords{equidistribution, compact K\"ahler manifold, holomorphic endomorphism, positive closed current}

\begin{document}
\begin{abstract}
	We prove that for a surjective holomorphic endomorphism $f$ of a compact K\"ahler manifold $X$ of dimension $k\ge 2$ and for some integer $p$ with $1\le p\le k$, there exists a proper invariant analytic subset $E$ for $f$ such that if a positive closed $(p, p)$-current $S$ can be represented by a smooth form in a neighborhood of $E$, the sequence $d_p^{-n}(f^n)^*(S-\alpha_S)$ converges to $0$ exponentially fast in the sense of currents, where $d_p$ is the dynamical degree of order $p$ and $\alpha_S$ is a smooth closed $(p, p)$-form in the de Rham cohomology class of $S$.
\end{abstract}
\maketitle

\section{Introduction}
Let $(X, \omega)$ be a compact K\"ahler manifold of complex dimension $k\ge 2$ such that $\int_X\omega^k=1$. Let $f:X\to X$ be a surjective holomorphic map. For $0\le p\le k$, the dynamical degree $d_p$ of order $p$ of $f$ is the spectral radius of the pull-back operator $f^*$ acting on the Hodge cohomology group $H^{p, p}(X, \C)$. It is known that $d_p$ itself is an eigenvalue of $f^*$ on $H^{p, p}(X, \C)$. An inequality due to Khovanskii, Teissier and Gromov (\cite{DN}, \cite{Gromov}) implies that the function $p\to \log d_p$ is concave on $0\le p\le k$. In particular, there are integers $\pmain$ and $\pmain'$ with $0\le \pmain\le \pmain'\le k$ such that
\begin{displaymath}
	d_0<\cdots<d_{\pmain}=\cdots=d_{\pmain'}>\cdots>d_k.
\end{displaymath}
We always have $d_0=1$. The last dynamical degree $d_k$ is also called the topological degree of $f$ because it is equal to the cardinality of $f^{-1}(x)$ for a generic point $x$ in $X$. We call $d_{\pmain}$ the main dynamical degree of $f$.
\medskip

The aim of  this paper is to study the dynamics of holomorphic endomorphisms on compact K\"ahler manifolds by proving the following theorem:
\begin{theorem}
	\label{thm:main}
	Let $(X, \omega)$ be a compact K\"ahler manifold of dimension $k\ge 2$ and $f:X\to X$ a surjective holomorphic endomorphism. Let $1 \le p\le k$ be an integer such that $d_{p-1}<d_p$. Then, there exists a proper (possibly empty) analytic subset $E$ invariant under $f$ such that 
	if a positive closed $(p, p)$-current $S$ can be represented by a smooth form in a neighborhood of $E$, we have
	\begin{align*}
		d_p^{-n}(f^n)^*(S-\alpha_S) \to 0
	\end{align*}
	exponentially fast in the sense of currents where $\alpha_S$ is a smooth closed $(p, p)$-form in the de Rham cohomology class $\{S\}$ of $S$. More precisely, there exist constants $c>0$ and $0<\rho<1$ such that for a smooth test $(k-p, k-p)$-form $\varphi$, we have
	\begin{align*}
		\left|\langle d_p^{-n}(f^n)^*(S-\alpha_S), \varphi \rangle \right|\le c\rho^n\|\varphi\|_{C^2}.
	\end{align*}
\end{theorem}

The equidistribution of inverse images of a positive closed current under holomorphic or meromorphic endomorphisms has been extensively researched on the complex projective space $\P^k$. To list a few, \cite{Lyubich}, \cite{eigenvaluations}, \cite{FS-II}, \cite{DS05}, \cite{DS09}, \cite{DS10}, \cite{Taflin}, \cite{Ahn21} and references therein.
\smallskip

However, it has not been studied much on compact K\"ahler manifolds. For meromorphic maps and for $p=k$, see \cite{DNT}. For holomorphic endomorphisms, non-pluripolar products were considered in \cite{AV}. For holomorphic automorphisms, see \cite{DS10-1}, \cite{dTD}. See also \cite{LV}, \cite{Vergamini}. Notice that our theorem works for general bidegrees and also for holomorphic automorphisms as well.
\smallskip

In \cite{Ahn16}, the following was proved on $\P^k$:
\begin{theorem}[Theorem 1.3 in \cite{Ahn16}]\label{thm:main_ahn16}
	Let $f:\P^k\to\P^k$ be a surjective holomorphic endomorphism of algebraic degree $d\ge 2$.
	Then, there is a proper (possibly empty) invariant analytic subset $E$ for $f$ such that if a positive closed $(p, p)$-current $S$ can be represented by a smooth form in a neighborhood of $E$,
	$d^{-pn}(f^n)^*(S)$ converges to $c_ST^p$ exponentially fast in the sense of currents where $c_S=\langle S, \omega^{k-p}\rangle$, $T^p=\lim_{n\to\infty}d^{-pn}(f^n)^*(\omega^p)$ is the Green current of order $p$ associated with $f$, and $\omega$ is the Fubini-Study form normalized so that $\int_{\P^k}\omega^k=1$.
\end{theorem}
One important property of a holomorphic endomorphism $f$ of $\P^k$ is that the graph of $f^{-1}$ is a holomorphic correspondence. Intuitively, if a point is not in the Julia set, its forward images accumulate near the attracting set under iteration. So, if a point is not trapped in the attracting set, iterates push its inverse images away from the attracting set to the boundary of the Fatou set, which is in the Julia set. Since $f^{-1}$ is a holomorphic correspondence, once the initial current has no mass on the attracting set, then its inverse images never have mass on the attracting set due to the work in \cite{DS07}. The condition in Theorem \ref{thm:main_ahn16} is a sufficient condition that the initial current has no mass on the attracting set. Therefore, all the mass moves towards the Julia set. Theorem \ref{thm:main_ahn16} is a quantitative explanation of this phenomenon. Technically, this property appears in the availability of the Lojasiewicz type inequality as in Lemma 3.3 in \cite{Ahn16}. Then, using the Lojasiewicz type inequality, we can control the influence from the set $E$, which appears in the H\"older continuity of the quasi-potential of $f_*(\omega)$.
\smallskip

On a general compact K\"ahler manifold, \cite[Lemma 4.7]{DNV} guarantees that the same is true and motivates us to generalize the work in \cite{Ahn16} to general compact K\"ahler manifolds.
\begin{lemma}
	[Lemma 4.7 in \cite{DNV}] Let $f$ be a surjective holomorphic map from $X$ to $X$. Then $f^{-1}$ is a holomorphic correspondence.
\end{lemma}

However, there are differences between projective spaces and compact K\"ahler manifolds. Firstly, general compact K\"ahler manifolds may not have many automorphisms and therefore, it is not clear that positive closed currents can be approximated by smooth positive closed currents. Secondly, the Green potential (see Subsection \ref{sec:superpotentials} for the definition) may not be negative. In addition, the cohomology groups of a general compact K\"ahler manifold are not simple compared to those of a complex projective space. Lastly, the existence of the Green current is not clear.
\medskip

In this work, for the lack of symmetry, we quantify the approximation theorem of positive closed currents introduced by Dinh-Sibony in \cite{DS04} and \cite{DS10-1} so as to control the $C^1$-norm of approximating smooth closed currents. Compared to the case of $\P^k$, approximating smooth closed currents may not be positive. So, we need additional estimates in the form of (semi-)regular transforms. In this work, we use the Green potential kernel induced from the work by Bost-Gillet-Soul\'e in \cite{BGS}. In general, the Green potential is not negative and therefore, we had to extend the estimates for the Green quasi-potential in \cite{DS09} and \cite{Ahn16} to (semi-)regular transforms of positive closed currents. Since the cohomology groups may be complicated, the mass of a positive closed current may not behave as nicely as in the case of $\P^k$. So, we normalize the inverse image of a positive closed current with the dynamical degree. Since the existence of the Green current is not clear, we rather observe the current $S-\alpha_S$ instead of $S$ and later in Section \ref{sec:Green}, we consider sufficient conditions for the Green current to exist.
\medskip

Summarizing the discussion in Section \ref{sec:Green}, we obtain the following:



\begin{theorem}\label{thm:almost_simple}
	Let $(X, \omega)$ be a compact K\"ahler manifold of dimension $k\ge 2$ and $f:X\to X$ a surjective holomorphic endomorphism such that that $d_{p-1}<d_p$ and $d_p$ is a simple eigenvalue of $f^*:H^{p, p}(X, \C)\to H^{p, p}(X, \C)$ and that other eigenvalues of $f^*$ on $H^{p, p}(X, \C)$ have modulus strictly less than $d_p$.
	Then, the Green current of order $p$, that is, $T^+_p:=\lim_{n\to\infty}d_p^{-n}(f^n)^*\omega^p$, exists and there exists a proper analytic subset $E$ invariant under $f$ such that for every positive closed $(p, p)$-current $S$ smooth in a neighborhood of $E$, we have
	\begin{align*}
		d_p^{-n}(f^n)^*S \to c_ST^+_p
	\end{align*}
	exponentially fast in the sense of currents where $T^+_p:=\lim_{n\to\infty}d_p^{-n}(f^n)^*\omega^p$ is the Green current of order $p$, 
	and $c_S=\lim_{n\to\infty}\frac{ (f^n)^*\{S\}\smile\{\omega^{k-p}\}}{d_p^{n} \{T^+_p\}\smile\{\omega^{k-p}\}}$. In particular, when the action of $f^*$ on cohomology is simple and $p$ is the order of the main dynamical degree ($p=\pmain$), we have $c_S=\frac{\left\langle S, T_p^-\right\rangle}{\left\langle \omega^p, T_p^-\right\rangle}$ where $T_p^-:=\lim_{n\to\infty}d_p^{-n}(f^n)_*\omega^{k-p}$.
\end{theorem}

\begin{corollary}
	Assume $f$ and $E$ as in Theorem \ref{thm:almost_simple}. Then, if an analytic subset $H$ of pure codimension $(p, p)$ does not meet $E$, then \begin{align*}
		d_p^{-n}(f^n)^*[H] \to c_HT^+_p
	\end{align*}
	exponentially fast in the sense of currents where $[H]$ denotes the current of integration on $H$ and $c_H=\lim_{n\to\infty}\frac{ (f^n)^*\{[H]\}\smile\{\omega^{k-p}\}}{d_p^{n} \{T^+_p\}\smile\{\omega^{k-p}\}}$.
\end{corollary}
The condition in Theorem \ref{thm:almost_simple} means a version of hyperbolicity that there is only one maximally increasing direction in $H^{p,p}(X, \C)$ for the operator $d_p^{-1}f^*$ and in all the other directions, it is decreasing. When $X=\P^k$, we have $d_p=d^p$ for $1\le p\le k$. Thus, Theorem \ref{thm:almost_simple} implies Theorem \ref{thm:main_ahn16}.
\smallskip

One remark is that while our primary interest is on holomorphic endomorphisms, Theorem \ref{thm:main} applies to holomorphic automorphisms as well. Based on the proof, if $f$ is a holomorphic automorphism of $X$, then $E=\emptyset$, which means that every positive closed current satisfies Theorem \ref{thm:main}. Some examples of this are holomorphic automorphisms of hyperk\"ahler manifolds with positive entropy in \cite{oguiso}. Concerning Theorem \ref{thm:almost_simple}, in the case of holomorphic automorphisms, \cite{dTD} obtained finer results without speed of convergence. See also \cite{DS10-1}.
\smallskip

Another remark is that in this context, it is reasonable to consider a version of the Dinh-Sibony conjecture \cite[Conjecture 1.4]{DS08} as follows and our work partially answers Question in the non-intersecting case.
\begin{question} Let $(X, \omega)$ be a compact K\"ahler manifold of dimension $k\ge 2$ and $f:X\to X$ a surjective holomorphic endomorphism with simple action on cohomology. Let $p$ be the order of the main dynamical degree ($p=\pmain$) and $T_p^\pm$ the associated Green currents. Then, for a generic analytic subset $H$ of $X$ of pure dimension $k-p$, $d_{p}^{-n}(f^n)^*[H]$ converges to $c_HT_p^+$ in the sense of currents, where $c_H=\frac{\langle [H], T_p^-\rangle}{\langle \omega^p, T_p^-\rangle}$. Here, $H$ is generic if either $H\cap \mathcal{E}=\emptyset$ or $\codim H\cap \mathcal{E}=p+\codim \mathcal{E}$ for any irreducible component $\mathcal{E}$ of every totally invariant proper analytic subset of $X$. 
\end{question}

\noindent
\textbf{Notation.} We denote by $\Phi_n$ and $\Psi_n$ the hypersurface of the critical points and that of the critical values of $f^n$, respectively for $n=1, 2, \cdots$. For an analytic subset $V$ of $X$, $[V]$ means the current of integration over the regular part of $V$ and $\mathrm{multi}_x V$ means the multiplicity of $V$ at $x\in V$ as an analytic subset. For a positive closed current $R$, $\nu(x, R)$ means the Lelong number of $R$ at $x\in X$. 
\medskip

The distance $\dist(\cdot, \cdot)$ on a compact K\"ahler manifold means the distance with respect to the natural metric associated with the K\"ahler form. For a subset $A$ of $X$ and a constant $\theta>0$, the $\theta$-neighborhood $A_\theta$ of $A$ denotes the set of points $x\in X$ such that $\dist(x, A)<\theta$. The norms $\|\cdot\|_\infty$, $\|\cdot\|_{C^\alpha}$ and $\|\cdot\|_{L^\alpha}$ (or $\|\cdot\|_{\infty, U}$, $\|\cdot\|_{C^\alpha, U}$ and $\|\cdot\|_{L^\alpha, U}$ for a subset $U$ of $X$) of a function or a form are the norms of the function or the sum of the corresponding norms of its coefficients (on $U$) with respect to a fixed finite atlas of $X$, respectively.
\bigskip

\noindent
\textbf{Acknowledgments.} 
The research of the author was supported by the National Research Foundation of Korea (NRF) grant funded by the Korea government (MSIT) (No. RS-2023-00250685). The paper was partially prepared during the visit of the author at the University of Michigan, Ann Arbor. He would
like to express his gratitude to the organization for hospitality. The author would like to thank Dan Burns for constructive discussions and for his support.

\section{Preliminaries}
\subsection{Currents}\label{sec:currents}
For the basics of currents, we refer the reader to \cite{Demailly}. In this section, we introduce some notations that we will use.
\medskip

Let $p\in\{1, 2, \ldots, k\}$. For a positive/negative $(p, p)$-current $S$ on $X$, we define the mass of $S$ by
\begin{displaymath}
	\|S\|:=|\langle S, \omega^{k-p}\rangle|.
\end{displaymath}
Let $\Cc_p$ denote the cone of positive closed $(p, p)$-currents on $X$, $\Dc_p$ the real vector space generated by $\Cc_p$ and $\Dc_p^0$ the space of currents $S\in\Dc_p$ such that $\{S\}=0$ in $H^{p,p}(X, \R)$. The duality between the cohomology groups implies that if $S$ is a current in $\Cc_p$, its mass depends only on the class $\{S\}$ in $H^{p,p}(X, \R)$. We define the $*$-norm $\|\cdot\|_*$ on $\Dc_p$ by
\begin{displaymath}
	\|S\|_*:=\min (\|S^+\|+\|S^-\|),
\end{displaymath}
where the minimum is taken over $S^\pm\in\Cc_p$ such that $S=S^+-S^-$. A subset in $\Dc_p$ is said to be $*$-bounded if it is bounded with respect to the $*$-norm. We will use the following $*$-topology on $\Dc_p$ and $\Dc_p^0$. We say that $S_n$ converges to $S$ in $\Dc_p$ if $S_n\to S$ weakly as currents and if the set $\{\|S_n\|_*\}$ is bounded. Note that if we restrict the $*$-topology to a $*$-bounded subset of $\Dc_p$, then it coincides with the usual weak topology on the space of currents. According to \cite{DS10-1}, smooth forms are dense in $\Dc_p$ and $\Dc_p^0$ with respect to the $*$-topology. We denote by $\widetilde{\Dc_p}$ and $\widetilde{\Dc_p^0}$ the subsets of smooth forms in $\Dc_p$ and $\Dc_p^0$, respectively.
\medskip

We will also work on the space of $L^1$-functions which can be written as a difference of two quasi-plurisubharmonic (q-psh for short) functions on $X$. On that space, we will use the DSH-norm which will be denoted by $\|\cdot\|_\dsh$ and defined by
\begin{align*}
	\|f\|_\dsh:=\|f\|_{L^1}+\|dd^c f\|_*.
\end{align*} 

There are also some other natural norms and distances on $\Dc_p$ which are closely related to the weak topology. 
For $\alpha>0$, if $S$ and $S'$ are currents in $\Dc_p$, we define
\begin{align*}
	\|S\|_{C^{-\alpha}}:=\sup_{\|\varphi\|_{C^\alpha}\le 1}|\langle S, \varphi\rangle|\quad\textrm{ and }\quad \dist_\alpha(S, S'):=\|S-S'\|_{C^{-\alpha}},
\end{align*}
where the supremum is taken over the set of the smooth test $(k-p, k-p)$-forms $\varphi$ on $X$ with $\|\varphi\|_{C^\alpha}\le 1$.
Observe that $\|\cdot\|_{C^{-\alpha}}\lesssim \|\cdot\|_*$ for every $\alpha>0$. 
\medskip

Thanks to the standard theory of interpolation between Banach spaces, we have
\begin{proposition}[Section 2.1 in \cite{DS09}] Let $\alpha$ and $\alpha'$ be strictly positive real numbers with $\alpha<\alpha'$. Then on any $*$-bounded subset of $\Dc_p$, the topology induced from $\dist_\alpha$ or from $\dist_{\alpha'}$ coincides with the weak topology. Moreover, on any $*$-bounded subset of $\Dc_p$, there is a constant $c_{\alpha,\alpha'}>0$ such that
	\begin{align*}
		\dist_{\alpha'}\le \dist_\alpha\le c_{\alpha,\alpha'}[\dist_{\alpha'}]^{\alpha/\alpha'}.
	\end{align*}
	In particular, a function on a $*$-bounded subset of $\Dc_p$ is H\"older continuous with respect to $\dist_\alpha$ if and only if it is H\"older continuous with respect to $\dist_{\alpha'}$.
\end{proposition}

\subsection{Semi-regular transforms}\label{sec:reg_transform}

We recall semi-regular transforms of currents in \cite{DS10-1}.
\medskip

Consider the compact K\"ahler manifold $\mathfrak{X}:=X\times X$ and $\Delta$ the diagonal submanifold of $\Xf$. We use $(x, y)$ to denote a point in $\Xf$. Let $\pi_i: \mathfrak{X}\to X$ denote the canonical projection on its factor for $i=1, 2$. Then, $\omega_\mathfrak{X}:=\pi_1^*\omega + \pi_2^*\omega$ is a natural K\"ahler form on $\mathfrak{X}$. Let $\pi:\widehat{\mathfrak{X}}\to \mathfrak{X}$ be the blow-up of $\mathfrak{X}$ along $\Delta$ in $\mathfrak{X}$ and let $\widehat{\Delta}:=\pi^{-1}(\Delta)$ denote the exceptional hypersurface. We define $\Pi_i:=\pi_i\circ\pi$ for $i=1, 2$. Then, $\Pi_i$ and its restrcition to $\widehat{\Delta}$ are both submersions for $i=1, 2$. By a theorem of Blanchard in \cite{Bl}, $\Xfh$ is a compact K\"ahler manifold. We fix a K\"ahler form $\omega_{\Xfh}$ on $\widehat{X\times X}$ throughout the article. For later use in Section \ref{sec:regularization}, we assume that $\omega_{\Xfh}$ is normalized as follows. The current $\pi_*\left([\widehat{\Delta}]\wedge\omega_{\Xfh}^{k-1}\right)$ has support in $\Delta$ and is a positive closed current of bidimension $(k, k)$. Hence, by the support theorem, $\pi_*\left([\widehat{\Delta}]\wedge\omega_{\Xfh}^{k-1}\right)$ is a positive constant multiple of $[\Delta]$. By multiplying a proper positive constant to $\omega_{\Xfh}$, we may assume that $\pi_*\left([\widehat{\Delta}]\wedge\omega_{\Xfh}^{k-1}\right)=[\Delta]$.
\medskip

\begin{definition}
	Let $q\in\{0, 1, \ldots, k\}$. Let $\mathcal{Q}$ be a $(q, q)$-form on $\widehat{\mathfrak{X}}$ which is smooth outside $\widehat{\Delta}$ and such that there exists a constant $c_\Qc>0$ satisfying
	\begin{align*}
		|\mathcal{Q}|\le -c_\Qc\log\dist(\cdot,\widehat{\Delta})\quad\textrm{ and }\quad|\nabla \mathcal{Q}|\le c_\Qc\dist(\cdot, \widehat{\Delta})^{-1}
	\end{align*}
	in a neighborhood of $\widehat{\Delta}$. Let $p\in \{k-q, k-q+1, \ldots, k\}$. A linear mapping $\Lc^\mathcal{Q}$ on the space of currents of bidegree $(p,p)$ on $X$ to the space of $(p+q-k, p+q-k)$-currents on $X$ defined by
	\begin{align*}
		\Lc^\mathcal{Q}(S):=(\Pi_2)_*(\Pi_1^*(S)\wedge\mathcal{Q})
	\end{align*}
	is called a semi-regular transform of bidegree $(q-k, q-k)$ associated with the form $\mathcal{Q}$.
	\medskip
	
	If $\mathcal{Q}$ is smooth, then the transform $\Lc^\mathcal{Q}$ is said to be regular. If $\mathcal{Q}$ is positive, then the transform $\Lc^\mathcal{Q}$ is said to be positive. If $\mathcal{Q}$ is closed, then the transform $\Lc^\mathcal{Q}$ is said to be closed.
\end{definition}
Here, $|\mathcal{Q}|$ and $|\nabla \mathcal{Q}|$ mean the sum of the absolute value of the coefficients of $\mathcal{Q}$ and the sum of the estimates of the gradients of their coefficients with respect to a fixed finite atlas of $\widehat{\mathfrak{X}}$, respectively.

\begin{remark}
	A positive semi-regular transform maps positive currents to positive currents. A closed semi-regular transform $\Lc^\mathcal{Q}$ maps closed currents to closed currents and satisfies $\Lc^\mathcal{Q}(dd^c S)=dd^c\Lc^\mathcal{Q}(S)$ for every current $S$.
\end{remark}

Young's inequality (or H\"older's inequality) gives the following proposition.
\begin{proposition}
	[Lemma 2.1 in \cite{DS04} or Proposition 2.3.2 in \cite{DS10}]\label{prop:regularization_principle} Any semi-regular transform can be extended to a linear continuous operator from the space of currents of order $0$ to the space of $L^{1+1/k}$-forms. It defines a linear continuous operator from the space of $L^\alpha$-forms, $\alpha\ge 1$, to the space of $L^{\alpha'}$-forms where $\alpha'$ is given by $(\alpha')^{-1}+1=\alpha^{-1}+(1+1/k)^{-1}$ if $\alpha<k+1$ and $\alpha'=\infty$ if $\alpha\ge k+1$. It also defines a linear continuous operator from the space of $L^\infty$-forms to the space of $C^1$-forms.
\end{proposition}

\subsection{Superpotentials}\label{sec:superpotentials}
For superpotentials on compact K\"ahler manifolds, we refer the reader to \cite{DS10-1}. See also \cite{DS09} for the theory on complex projective spaces.
\medskip

Let $\beta=\{\beta_1, \cdots, \beta_{h_p}\}$ with $h_p=\dim H^{p,p}(X, \R)$  be a fixed family of real smooth closed $(p,p)$-forms such that the family of classes $\{\beta\}:=\{\{\beta_1\}, \cdots, \{\beta_{h_p}\}\}$ is a basis of $H^{p,p}(X, \R)$. We can also find another family $\beta^*=\{\beta^*_1, \cdots, \beta^*_{h_p}\}$ of real smooth closed $(k-p, k-p)$-forms so that $\{\beta^*\}:=\{\{\beta^*_1\}, \cdots, \{\beta^*_{h_p}\}\}$ is a basis of $H^{k-p, k-p}(X, \R)$ and is the dual basis of $\{\beta\}$ with respect to the cup-product $\smile$.
\medskip

Let $R$ be a current in $\Dc_{k-p+1}^0$ and $U_R$ a potential of $R$, that is, a $(k-p, k-p)$-current such that $dd^cU_R=R$. Adding to $U_R$ a suitable combination of $\beta^*_i$'s allows us to assume that $\langle U_R, \beta_i\rangle=0$ for $i=1, \cdots, h_p$. We say that $U_R$ is $\beta$-normalized.
\medskip

\begin{definition}[Definition 3.2.2 in \cite{DS10-1}]
	Let $\beta$ and $\beta^*$ be given families of real smooth closed forms of bidegree $(p, p)$ and $(k-p, k-p)$ as above, respectively. Let $S\in\Dc_p$. The $\beta$-normalized superpotential $\Uc_S$ of $S$ is the function defined on $\widetilde{\Dc_{k-p+1}^0}$ by
	\begin{align*}
		\Uc_S(R):=\langle S, U_R\rangle\,\, \textrm{ for }\, R\in \widetilde{\Dc_{k-p+1}^0},
	\end{align*}
	where $U_R$ is an $\beta$-normalized smooth potential of $R$. We say that $S$ has a bounded superpotential if $\Uc_S$ is bounded on each $*$-bounded subset of $\widetilde{\Dc_{k-p+1}^0}$. We say that $S$ has a continuous superpotential if $\Uc_S$ can be extended to a continuous function on $\Dc_{k-p+1}^0$. Here, the topology is with respect to the norm $\|\cdot\|_{C^{-\alpha}}$ for some $\alpha>0$. In this case, the extension is also denoted by $\Uc_S$ and is also called a $\beta$-normalized superpotential of $S$. We say that $S$ has a H\"older continuous superpotential if $\Uc_S$ is continuous and H\"older continuous on $\Dc_{k-p+1}^0$ with respect to the norm $\|\cdot\|_{C^{-\alpha}}$ for some $\alpha>0$.
\end{definition}

\begin{remark}
	The definition of the $\beta$-normalized superpotential $\Uc_S$ of $S$ is independent of the choice of $\beta^*$. If $S$ belongs to $\Dc_p^0$, the superpotential is independent of the choice of the family $\beta$.
	So, when we are dealing with superpotentials of currents $S\in\Dc_p^0$, we will not specify the families $\beta$ and $\beta^*$ and simply call $\Uc_S$ the superpotential of $S$.
\end{remark}

We recall the Green potential kernel and the Green potential of $S$ for $S\in\Dc_p^0$ in \cite{DS10-1}, which are useful in computing superpotentials. For reader's convenience, we summarize the construction.
\medskip

The integration on the diagonal submanifold $\Delta$ of $\mathfrak{X}$ defines a positive closed $(k, k)$-current $[\Delta]$. By the K\"unneth formula, we have a canonical isomorphism
\begin{align*}
	H^{k, k}(\Xf, \C)\simeq \sum_{0\le i\le k} H^{i, k-i}(X,\C)\otimes H^{k-i, i}(X, \C).
\end{align*}
Then, $[\Delta]$ is cohomologous to a smooth real closed $(k ,k)$-form $\alpha_\Delta$ which is a finite combination of forms of type $\pi_1^*(\psi)\wedge \pi_2^*(\psi')$ where $\psi$ and $\psi'$ are closed forms on $X$ of bidegree $(i, k-i)$ and $(k-i, i)$, respectively. So, $\alpha_\Delta$ satisfies $d_x\alpha_\Delta=d_y\alpha_\Delta=0$. Replacing $\alpha_\Delta(x, y)$ by $[\alpha_\Delta(x, y)+\alpha_\Delta(y,x)]/2$, we may assume that $\alpha_\Delta$ is symmetric, i.e. invariant under the involution $(x, y)\to (y, x)$ where $(x, y)$ denotes the coordinates of $\Xf$.
\medskip

According to \cite{BGS}, there is a real smooth closed $(k-1, k-1)$-form $\eta$ on $\Xfh$ such that $\pi^*(\alpha_\Delta)$ is cohomologous to $[\widehat{\Delta}]\wedge \eta$, where $[\widehat{\Delta}]$ is the positive closed $(1, 1)$-current of integration on $\widehat{\Delta}$. Hence, $\pi_*([\widehat{\Delta}]\wedge \eta)$ is cohomologous to $\alpha_\Delta$ and therefore to $[\Delta]$. Since the map $(x, y)\to (y, x)$ induces an involution on $\Xfh$, we can also choose $\eta$ symmetric with respect to this involution.
\medskip

Let $\alpha_{\widehat \Delta}$ be a real closed $(1, 1)$-form on $\Xfh$, which is cohomologous to $[\widehat{\Delta}]$. We can choose $\alpha_{\widehat \Delta}$ symmetric. There is a q-psh function $u$ on $\Xfh$ such that $dd^cu=[\widehat{\Delta}]-\alpha_{\widehat \Delta}$. This function is necessarily symmetric. Subtracting from $u$ a constant, we may assume that $u<-2$.\medskip

Choose a real smooth $(k-1, k-1)$-form $\delta_{\Xfh}$ on $\Xfh$ such that $dd^c\delta_\Xfh= \pi^*(\alpha_\Delta) -\alpha_{\widehat\Delta}\wedge\eta$ where $\alpha_\Delta$, $\alpha_{\widehat \Delta}$ and $\eta$ are as above. We can choose $\delta_\Xfh$ to be symmetric. Let $\Lc^K$ be the semi-regular transform associated with the form $K:=u \eta -\delta_\Xfh$. The we have the following proposition.
\begin{proposition}[Proposition 2.1 in \cite{DS10-1}]\label{prop:Green_kernel}
	Let $S\in\Dc_p^0$ for $p\ge 1$. Then $U_S:=\Lc^K(S)$ is a potential of $S$, that is, $dd^cU_S=S$ in the sense of currents. Moreover, we have
	\begin{align*}
		\|U_S\|_{L^{1+1/k}}\le c\|S\|_*
	\end{align*}
	for some constant $c>0$ independent of $S$.
\end{proposition}

We will call $U_S:=\Lc^K(S)$ the Green potential of $S$. Notice that $K$ is not necessarily positive. However, there exists a constant $m_K>0$ such that $\eta+m_K\omega_{\Xfh}^{k-1}$ and $\delta_\Xfh+m_K\omega_{\Xfh}^{k-1}$ are both strictly positive. We define the forms $K_\pm$ to be
\begin{align*}
	K_+:= u(\eta+m_K\omega_{\Xfh}^{k-1})-(m_K\omega_{\Xfh}^{k-1})\quad\textrm{ and }\quad K_-:= u(m_K\omega_{\Xfh}^{k-1})-(\delta_\Xfh + m_K\omega_{\Xfh}^{k-1}).
\end{align*}
Both $K_\pm$ are strictly negative and $\Lc^K=\Lc^{K_+}-\Lc^{K_-}$. In particular, the transforms $\Lc^{K_\pm}$ associated with the forms $K_\pm$ are negative.
%
%

\section{Regularization of positive closed currents}\label{sec:regularization}

The regularization of positive closed currents in this section was introduced in \cite{DS04} and \cite{DS10-1}. The main purpose of this section is to study its quantitative aspects.
\medskip

Let $u$ be the q-psh function on $X$ as in Subsection \ref{sec:superpotentials}. Let $\chi$ be a smooth convex increasing function on $\R\cup \{-\infty\}$ such that $\chi(t)=t$ for $t\ge 0$, $\chi(t)=-1$ for $t\le -2$ and $0\le \chi'\le 1$. For $\theta\in\C^*$, we define
\begin{align*}
	\chi_\theta(t):=\chi(t-\log|\theta|)+\log|\theta|\quad \textrm{ and }\quad u_\theta:=\chi_\theta(u).
\end{align*}
Then, $u_\theta = u_{|\theta|}$ and $u_\theta$ decreasingly converges to $u$ as $|\theta|\to 0$. Let $m_{\widehat{\Delta}}>0$ be so that $m_{\widehat{\Delta}}\omega_{\Xfh}-\alpha_{\widehat\Delta}$ is positive. Then, we have
\begin{align*}
	dd^cu_\theta&=(\chi_\theta''\circ u)du\wedge d^cu+(\chi_\theta'\circ u)dd^cu=-(\chi_\theta'\circ u)\alpha_{\widehat\Delta}\geq -m_{\widehat \Delta}\omega_{\Xfh}.
\end{align*}
So, for $\theta\in\C^*$, the smooth closed $(1, 1)$-current $\alpha_{\widehat \Delta}+dd^cu_\theta$ can be written as a difference of two smooth positive closed $(1, 1)$-currents as $\alpha_{\widehat \Delta}+dd^cu_\theta=(m_{\widehat \Delta}\omega_{\Xfh}+dd^cu_\theta)-(m_{\widehat \Delta}\omega_{\Xfh}-\alpha_{\widehat \Delta})$. So, $\big\|\alpha_{\widehat{\Delta}}+dd^cu_\theta\big\|_*$ is uniformly bounded in $\theta$.
\medskip

Recall that $\pi_*\left([\widehat{\Delta}]\wedge\omega_{\Xfh}^{k-1}\right)=[\Delta]$ by our choice of $\omega_{\Xfh}$. Let $\theta\in\C^*$ be such that $|\theta|<1$. For notational convenience, we write
\begin{align*}
	\Lc_\theta^+:=\Lc^{(m_{\widehat \Delta}\omega_{\Xfh}+dd^cu_\theta)\wedge\omega_{\Xfh}^{k-1}},\quad	\Lc^-:=\Lc^{(m_{\widehat \Delta}\omega_{\Xfh}-\alpha_{\widehat \Delta})\wedge\omega_{\Xfh}^{k-1}}\quad\textrm{ and }\quad \Lc_\theta:=\Lc^{(\alpha_{\widehat \Delta}+dd^cu_\theta)\wedge\omega_{\Xfh}^{k-1}}.
\end{align*}
Note that we can choose all the forms and functions involved in $\Lc_\theta^+$, $\Lc^-$ and $\Lc_\theta$ symmetric. So, for a $(p, p)$-current $S$ and a smooth $(k-p, k-p)$-form $\varphi$, we have $\langle \Lc_\theta^+(S), \varphi \rangle=\langle S, \Lc_\theta^+(\varphi) \rangle$ and some other relationships of this sort. We obviously have $\Lc_\theta=\Lc_\theta^+-	\Lc^-$ and $\Lc_\theta\to \Lc_0 := \id$ as $\theta\to 0$ on smooth forms (see \cite[p.486]{DS04}). The transforms $\Lc_\theta^+$ and $\Lc^-$ are positive closed and of bidegree $(0, 0)$.
\medskip

Our regularizing kernel $(\alpha_{\widehat \Delta}+dd^cu_\theta)\wedge\omega_{\Xfh}^{k-1}$ is slightly different from the one in \cite{DS10-1} but it works in the same way as in \cite[Lemma 2.4.6]{DS10-1}. Also, in \cite[Lemma 2.4.6]{DS10-1}, we do not actually need the closedness of a smooth current. 
%
So, we get the following lemma.
\begin{lemma}
	\label{lem:reg_smooth_estimate}Let $\varphi$ be a smooth $(k-p, k-p)$-form. Then, for every $\theta\in\C$, $\Lc_\theta(\varphi)$ is smooth and
	\begin{align*}
		\|\Lc_\theta(\varphi)-\varphi\|_\infty\le c|\theta|\|\varphi\|_{C^1}
	\end{align*}
	where $c>0$ is a constant independent of $\varphi$ and $\theta$.
\end{lemma}

From \cite[Lemma 2.4.1]{DS10-1}, we can deduce the following lemma. The constant $c_1>0$ below in Lemma \ref{lem:regularization_C1_dist} comes from the description of the support of $dd^cu_\theta$ in terms of $\theta$.
\begin{lemma}\label{lem:regularization_C1_dist}
	Let $\theta\in\C^*$ be such that $|\theta|\ll 1$. Then, there exist constants $c_1, c_2>0$ independent of $\theta$ such that for any subset $U$ of $X$ and for any $C^1$-form $S$ and $\supp S\Subset U$, 
	$\Lc_\theta(S)$ is a $C^1$ form with compact support in $U_{c_1\theta}$ such that
	\begin{align*}
		\|\Lc_\theta(S)\|_{C^1, U_{c_1\theta}}
		\le c_2\|S\|_{C^1, U}.
	\end{align*}
\end{lemma}

\begin{proof}
	The uniform norm estimate is quite straightforward. So, we omit it and focus on the estimate of its gradient. Since $S$ has $C^1$-coefficients, we have
	\begin{align*}
		\Lc_\theta(S)=\Lc^{(\alpha_{\widehat{\Delta}}+dd^cu_\theta)\wedge\omega_{\Xfh}^{k-1}}(S)&=(\Pi_2)_*(\Pi_1^*(S)\wedge (\alpha_{\widehat{\Delta}}+dd^c u_\theta)\wedge\omega_{\Xfh}^{k-1})\\
		&=(\pi_2)_*(\pi_1^*(S)\wedge \pi_*((\alpha_{\widehat{\Delta}}+dd^c u_\theta)\wedge \omega_{\Xfh}^{k-1})).
	\end{align*}
	According to \cite[Lemma 2.4.1]{DS10-1}, there exists a constant $c_1>0$ such that the support of $\pi_*((\alpha_{\widehat{\Delta}}+dd^c u_\theta)\wedge \omega_{\Xfh}^{k-1})$ is in $(\Delta)_{c_1|\theta|}$ for $\theta\in\C^*$ with $|\theta|\ll 1$. Also, notice that the support of $\pi_*(dd^c u_\theta)$ is away from $\Delta$. For this estimate, we do not need closedness. Applying a partition of unity to $S$, we may assume that the support of $S$ and the support of $\Lc_\theta(S)$ belong to the same coordinate neighborhood. Let $P\in X$. Let $(x, y)=(x_1, \cdots, x_k, y_1, \cdots, y_k)\in X\times X$ be local coordinates in a neighborhood of a point $(P, P)\in X\times X$ such that $(0, \cdots, 0, 0, \cdots, 0)$ corresponds to $(P, P)$ and that $S$ and $\Lc_\theta(S)$ are supported in the neighborhood when $|\theta|\ll 1$. Let $P'$ be a point in the support of $\Lc_\theta(S)$ and $y'=(y_1', \cdots, y_k')$ its coordinates.
	\medskip
	
	Let $v=(v_1, \cdots, v_k)\in \C^k$ be a direction vector such that $|v_1|^2+\cdots+|v_k|^2=1$. Then, we consider the derivative of $\Lc^{(\alpha_{\widehat{\Delta}}+dd^cu_\theta)\wedge\omega_{\Xfh}^{k-1}}(S)$ at $P'$ in the direction of $v$. We have
	{\begin{align}
			\frac{\Lc_\theta(S)(y'+hv)-\Lc_\theta(S)(y')}{h}	\label{eq:diff_quo}&=\frac{1}{h}\int_{x}[(\pi_1^*(S)(x)\wedge \pi_*((\alpha_{\widehat{\Delta}}+dd^c u_\theta)\wedge \omega_{\Xfh}^{k-1}))(x, y'+hv)\\
			\notag&\quad\quad\quad\quad\quad-\pi_1^*(S)(x)\wedge \pi_*((\alpha_{\widehat{\Delta}}+dd^c u_\theta)\wedge \omega_{\Xfh}^{k-1}))(x, y')].
	\end{align}}
	Since $S$ has $C^1$-coefficients, for
	\begin{align}
		\label{eq:S_C_1}&\frac{1}{h}\int_{x}[(\pi_1^*(S)(x+hv)\wedge \pi_*((\alpha_{\widehat{\Delta}}+dd^c u_\theta)\wedge \omega_{\Xfh}^{k-1}))(x, y')\\
		\notag&\quad\quad\quad\quad\quad-\pi_1^*(S)(x)\wedge \pi_*((\alpha_{\widehat{\Delta}}+dd^c u_\theta)\wedge \omega_{\Xfh}^{k-1}))(x, y')],
	\end{align} 
	$\lim_{h\to 0}\eqref{eq:S_C_1}$ exists and is a continuous form whose uniform norm is bounded by $\|S\|_{C^1}$ up to a multiplicative constant independent of $\theta$. Indeed, this is just the transform of the derivative of $S$ in the direction of $v$, which is a continuous form. So, it is enough to estimate
	\begin{align}
		\notag\eqref{eq:diff_quo}-\eqref{eq:S_C_1}&=\frac{1}{h}\int_{x}[(\pi_1^*(S)(x+hv)\wedge \pi_*((\alpha_{\widehat{\Delta}}+dd^c u_\theta)\wedge \omega_{\Xfh}^{k-1}))(x, y')\\
		\notag&\quad\quad\quad\quad\quad\quad\quad-\pi_1^*(S)(x)\wedge \pi_*((\alpha_{\widehat{\Delta}}+dd^c u_\theta)\wedge \omega_{\Xfh}^{k-1}))(x, y'+hv)]\\
		\notag&=\frac{1}{h}\int_{x}[(\pi_1^*(S)(x+hv)\wedge \pi_*((\alpha_{\widehat{\Delta}}+dd^c u_\theta)\wedge \omega_{\Xfh}^{k-1}))(x, y')\\
		\notag&\quad\quad\quad\quad\quad\quad\quad\quad-\pi_1^*(S)(x+hv)\wedge \pi_*((\alpha_{\widehat{\Delta}}+dd^c u_\theta)\wedge \omega_{\Xfh}^{k-1}))(x+hv, y'+hv)]\\
		\label{integral:5-1}&=\frac{1}{h}\int_{x}[(\pi_1^*(S)(x+hv)\wedge[ \pi_*((\alpha_{\widehat{\Delta}}+dd^c u_\theta)\wedge \omega_{\Xfh}^{k-1}))(x, y')\\
		\notag&\quad\quad\quad\quad\quad\quad\quad\quad- \pi_*((\alpha_{\widehat{\Delta}}+dd^c u_\theta)\wedge \omega_{\Xfh}^{k-1}))(x+hv, y'+hv)]]
	\end{align}
	for all $h\in \R$ with $0<|h|\ll 1$ and show that this quantity is bounded by $\|S\|_\infty$ up to a multiplicative constant independent of $\theta$. The second equality comes from the change of the variable $x\to x-hv$ inside the integral. Since $S$ is bounded, it suffices to consider the estimate of 
	\begin{align}\label{qty:kernel_diff}
		\frac{1}{h}[ \pi_*((\alpha_{\widehat{\Delta}}+dd^c u_\theta)\wedge \omega_{\Xfh}^{k-1}))(x, y')- \pi_*((\alpha_{\widehat{\Delta}}+dd^c u_\theta)\wedge \omega_{\Xfh}^{k-1}))(x+hv, y'+hv)].
	\end{align}
	We change variables $(x, y)\to (x, z:=y-x)$. Then, in the neighborhood of $(P, 0)\in\{(x, z)\}$, $\Delta=\{z=0\}$ and the operator $(\pi_2)_*$ is the integration with respect to the $x$-variable. Since the support of $\pi_*(dd^c u_\theta)$ is away from $\Delta$, it is enough to consider the case where one of $z_i$'s is not equal to $0$. Without loss of generality, we may consider a sector $\{0<|z_1|<1, |z_i|<2|z_1|\}$. Then, once again, we change variables $(x, z)\to (x, z_1, z_2/z_1, \cdots, z_k/z_1)$ in that neighborhood. We denote by $w_1=z_1$ and $w_i=z_i/z_1$ for $i=2, \cdots, k$. With respect to the coordinates $(x, w)$, we have
	\begin{align*}
		\eqref{qty:kernel_diff}=\frac{1}{h}[ \pi_*((\alpha_{\widehat{\Delta}}+dd^c u_\theta)\wedge \omega_{\Xfh}^{k-1})(x, w')- \pi_*((\alpha_{\widehat{\Delta}}+dd^c u_\theta)\wedge \omega_{\Xfh}^{k-1}))(x+hv, w')]
	\end{align*}
	where $w'_1=y'_1-x_1$ and $w'_i=(y'_i-x_i)/(y'_1-x_1)$ for $i=2, \cdots, k$. With respect to the coordinates $(x, w)$, the function $u$ can be written as $u=\log|w_1|+ \psi$ for a smooth function $\psi$ of $z$ and $w$. Then, we have
	\begin{align*}
		dd^c u_\theta &= dd^c [\chi_\theta (u)]=\chi_\theta''(u)d(\log|w_1|+ \psi)\wedge d^c(\log|w_1|+ \psi)+\chi_\theta'(u)dd^c(\log|w_1|+ \psi).
	\end{align*}
	Hence, as $h\to 0$, we get some extra derivatives of $\psi$ and $\chi_\theta(u)$ with respect to $x$, which are bounded independently of $\theta$. So, by the argument used in \cite[Lemma 2.4.1]{DS10-1}, $|\eqref{qty:kernel_diff}|$ is bounded by $1/|\theta|^2$ up to a multiplicative constant independent of $\theta$. The support of $dd^c(\alpha_{\widehat{\Delta}}+dd^cu_\theta)$ sits inside $\{|x_1-y_1'|<c'|\theta|\}$. When computing the push-forward $\pi_*$, we integrate a bounded integrand on $|x_1-y'_1|<c'|\theta|$ for some $c'>0$ independent of $\theta$. Then, it gives us a value  of order $|\theta|^2$ and cancels the factor $1/|\theta|^2$. So, the value $|\eqref{integral:5-1}|$ is bounded by $\|S\|_\infty$ up to a multiplicative constant independent of $\theta$ and $S$.
\end{proof}

The above two lemmas imply the following:
\begin{proposition}\label{prop:uniform_estimate_smooth}
	Let $\varphi$ be a smooth test $(k-p, k-p)$-form. Then, $\|\Lc_\theta^{k+2}(\varphi)-\varphi\|_\infty$ converges to $0$ as $\theta\to 0$.
\end{proposition}

\begin{proof}
	It is enough to observe that
	\begin{align*}
		\Lc_\theta^{k+2}(\varphi)-\varphi=\sum_{i=0}^{k+1}\Lc_\theta(\Lc_\theta^i(\varphi))-\Lc_\theta^i(\varphi)
	\end{align*}
	and that Lemma \ref{lem:regularization_C1_dist} implies $\|\Lc_\theta^i(\varphi)\|_{C^1}$ is uniformly bounded by $\|\varphi\|_{C^1}$ up to a constant independent of $\theta$ and $i=0, \cdots, k+1$. Then, Lemma \ref{lem:reg_smooth_estimate} completes the proof.
\end{proof}

For $S\in\Dc_p$, we define
\begin{align}\label{eq:theta_regularization}
	S_\theta&:=\Lc_\theta\circ\cdots\circ\Lc_\theta(S)=\Lc_\theta^{k+2}(S).
\end{align}
According to Proposition \ref{prop:regularization_principle}, $S_\theta$ is a current with $C^1$-coefficients. According to \cite{DS04}, $S_\theta\in\{S\}$. This regularization is essentially the same as the one in \cite{DS04} but slightly different from the one in \cite{DS10-1}. Since currents in $\Dc_s$ are of order $0$, Proposition \ref{prop:uniform_estimate_smooth} implies the convergence $S_\theta\to S$ in the sense of currents
\begin{corollary}
	For each current $S\in\Dc_p$, the sequence $(S_\theta)$ of $C^1$-smooth currents converges to $S$ in the sense of currents.
\end{corollary}

Note that $S_\theta$ is not positive in general. However, the regularization $S_\theta$ can be written as $S_\theta =\sum_i (-1)^{{\rm sgn}(i)}\Lc_{i, 1}\circ\cdots \circ \Lc_{i, k+2}(S)$ where $\Lc_{i,j}$ is either $\Lc_\theta^+$ or $\Lc^-$ and ${\rm sgn}(i)$ is the number of indices $\Lc_{i,j}$ such that $\Lc_{i,j}=\Lc^-$. Then, $S_\theta$ can be written as 
$$S_\theta=S_\theta^+-S_\theta^-$$
where $S_\theta^\pm$ are positive closed currents of bidegree $(p,p)$. Indeed, $S_\theta^+$ (or $S_\theta^-$) is just the sum of terms $(-1)^{{\rm sgn}(i)}\Lc_{i, 1}\circ\cdots \circ \Lc_{i, k+2}(S)$ with an even (or odd) number of $\Lc^-$'s. Proposition \ref{prop:regularization_principle} implies that $S_\theta^\pm$ are of $C^1$-coefficients. We give $C^1$-estimates of $S_\theta^\pm$ in terms of $\theta$.
\begin{proposition}\label{prop:C1_estimate_regularization}
	There exists a constant $c_\reg>0$ such that for all $\theta\in\C^*$ with $|\theta|\ll 1$ and for all $S\in\Cc_p$, we have
	\begin{align*}
		\|S^\pm_\theta\|\le c_\reg\|S\|\quad \textrm{ and }\quad\|S^\pm_\theta\|_{C^1}\le c_\reg|\theta|^{-(2k+2)(k+2)}\|S\|.
	\end{align*}
\end{proposition}

\begin{lemma}
	Let $S\in\Cc_p$. Then for all $\theta\in\C^*$ with $|\theta|\ll 1$, the current $\Lc^{dd^c u_\theta\wedge\omega_{\Xfh}^{k-1}}(S)$ is smooth and we have
	\begin{align*}
		\|\Lc^{dd^c u_\theta\wedge\omega_{\Xfh}^{k-1}}(S)\|_{C^1}\le c|\theta|^{-(2k+2)}\|S\|
	\end{align*}
	where $c>0$ is a constant independent of $S$ and $\theta$.
\end{lemma}


\begin{proof}
	The first assertion is from the observation that the support of $dd^c u_\theta$ is away from $\widehat{\Delta}$. So, we consider the second assertion. We have
	\begin{align*}
		\Lc^{dd^cu_\theta\wedge\omega_{\Xfh}^{k-1}}(S)&=(\Pi_2)_*(\Pi_1^*(S)\wedge dd^c u_\theta\wedge\omega_{\Xfh}^{k-1})\\
		&=(\pi_2)_*\left(\pi_1^*(S)\wedge \pi_*(dd^c u_\theta)\wedge \pi_*\left(\omega_{\Xfh}^{k-1}\right)\right).
	\end{align*}
	For the support of $dd^c u_\theta$ does not intersect $\widehat{\Delta}$ and $\pi$ is biholomorphic outside $\Delta$. We estimate $\pi_*(dd^c u_\theta)\wedge \pi_*\left(\omega_{\Xfh}^{k-1}\right)$.
	\medskip
	
	Since $X$ is compact and all the derivatives of $u_\theta$ are uniformly bounded outside a neighborhood of $\Delta$, it suffices to estimate in a neighborhood of $(P, P)\in \Delta \subset X\times X$ for $P\in X$. Let $(x, y)=(x_1, \cdots, x_k, y_1, \cdots, y_k)\in X\times X$ be local coordinates in a neighborhood of a point $(P, P)\in X\times X$ so that $(0, 0)$ corresponds to $(P, P)$. After shrinking the neighborhood of $(P, P)$ if necessary, we make a change of variables $(x, y)\to (x, z:=y-x)$. Then, in the neighborhood of $(P, P)$, $\Delta=\{z=0\}$ and the operator $(\pi_2)_*$ is the integration with respect to the $x$-variable.
	\medskip
	
	Since the support of $\pi_*\left(dd^c u_\theta\wedge\omega_{\Xfh}^{k-1}\right)$ is away from $\Delta$, it is enough to estimate the $C^1$-norm of $\pi_*\left(dd^c u_\theta\wedge\omega_{\Xfh}^{k-1}\right)$ in a sector $\mathfrak{S}$ of the form $|z_1|<\varepsilon$ and $|z_j|<2|z_1|$ where $z=(z_1, \cdots, z_k)$. Here, for simplicity, we just take $\varepsilon=1$. We apply the same arguments to other sectors.
	\medskip
	
	The following is a slight modification of \cite[Lemma 2.4.1]{DS10-1}. We again make a change of variables $w_1=z_1$ and $w_j=z_j/z_1$ for $j\ge 2$ so that $(x, w_1, w_2, \cdots, w_k)$ becomes the natural coordinates of $\pi^{-1}(\mathfrak{S})\subset \widehat{X\times X}$ and $\widehat{\Delta}=\{w_1=0\}$.
	\medskip
	
	As in the proof of \cite[Lemma 2.4.1]{DS10-1}, the support of $dd^c u_\theta$ sits inside $\{c_1|\theta|\le|w_1|\}$ for some constant $c_1>0$ independent of $\theta$. 
	As in Lemma \ref{lem:regularization_C1_dist}, with respect to the coordinates $(x, w)$, the function $u$ can be written as $u=\log|w_1|+ \psi$ for a smooth function $\psi$ of $x$ and $w$. Then, we have
	\begin{align*}
		dd^c u_\theta &= dd^c [\chi_\theta (u)]=\chi_\theta''(u)d(\log|w_1|+ \psi)\wedge d^c(\log|w_1|+ \psi)+\chi_\theta'(u)dd^c(\log|w_1|+ \psi).
	\end{align*}
	Observe that $\pi_*(dd^c u_\theta)$ is obtained from replacing $w_1, \cdots, w_k$ by $z_1, z_2/z_1, \cdots, z_k/z_1$. So, we have
	\begin{align*}
		&\pi_*(dd^c u_\theta)=\chi_\theta''(u)d(\log|z_1|+ \psi(x, z_1, z_2/z_1, \cdots, z_k/z_1))\wedge d^c(\log|z_1|+ \psi(x, z_1, z_2/z_1, \cdots, z_k/z_1))\\
		&\quad\quad\quad\quad\quad\quad+\chi_\theta'(u)dd^c(\log|z_1|+ \psi(x, z_1, z_2/z_1, \cdots, z_k/z_1)).
	\end{align*}
	
	Also, each component of $\omega_{\Xfh}^{k-1}$ is of the form $C(x, w)dx_I\wedge d\bar x_J\wedge d w_{I'}\wedge d\bar w_{J'}$ where $C(x, w)$ is a smooth function of $x$ and $w$, and $dx_I\wedge d\bar x_J\wedge d w_{I'}\wedge d\bar w_{J'}$ is the wedge product of $dx_1, \cdots, dx_k$ and $dw_1, \cdots, dw_k$ of bidegree $(k-1, k-1)$. Again, $\pi_*\left(\omega_{\Xfh}^{k-1}\right)$ is obtained from replacing $w_1, \cdots, w_k$ by $z_1, z_2/z_1, \cdots, z_k/z_1$. From direct computations, we have
	\begin{align*}
		d\left(\frac{y_i-x_i}{y_1-x_1}\right)=\frac{1}{y_1-x_1}d(y_i-x_i)-\frac{y_i-x_i}{(y_1-x_1)^2}d(y_1-x_1)
	\end{align*}
	for each $i=2, \cdots, k$. After changing coordinates back to $(x, y)$ from $(x, z)$, each coefficient of $\pi_*\left(\omega_{\Xfh}^{k-1}\right)$ is of the form $(y_1-x_1)^{-l} (\overline{y_1} - \overline{x_1})^{-l'} C'(x, y_1-x_1, (y_2-x_2)/(y_1-x_1), \cdots, (y_k-x_k)/(y_1-x_1))$ where $0\le l, l'\le k-1$ and $C'$ is a smooth function of $x, y_1-x_1, (y_2-x_2)/(y_1-x_1), \cdots, (y_k-x_k)/(y_1-x_1)$. The support of $dd^c u_\theta$ lies in $\{c_1|\theta|\le |z_1|\}$. The $C^1$-norm of each coefficient of $\pi_*\left(\omega_{\Xfh}^{k-1}\right)$ with respect to $y$ is bounded by $|\theta|^{-2k+1}$ up to a multiplicative constant independent of $\theta$. The proof of \cite[Lemma 2.4.1]{DS10-1} implies that the $C^1$-norm of the coefficients of $dd^c u_\theta$ on its support is bounded by $|\theta|^{-3}$ up to a multiplicative constant independent of $\theta$. So, we get
	\begin{align*}
		\left\|(\pi_2)_*\left(\pi_1^*(S)\wedge \pi_*(dd^c u_\theta)\wedge \pi_*\left(\omega_{\Xfh}^{k-1}\right)\right)\right\|_{C^1}\lesssim |\theta|^{-(2k+2)}\|S\|
	\end{align*}
	where the inequality $\lesssim$ means $\le$ up to a multiplicative constant independent of $S$ and $\theta$.
\end{proof}

\begin{proof}[Proof of Proposition \ref{prop:C1_estimate_regularization}]
	The first estimate is from the same argument used in \cite[Theorem 1.1]{DS04}. We consider the second inequality. Without loss of generality, we may assume that $S$ is a positive closed $(p, p)$-current on $X$. Notice that $\Lc_\theta^+=\Lc^{m_{\Xfh}\wedge\omega_{\Xfh}^k}+\Lc^{dd^c u_\theta\wedge\omega_{\Xfh}^{k-1}}$ and $\Lc^-=\Lc^{m_{\Xfh}\wedge\omega_{\Xfh}^k}-\Lc^{\alpha_{\Xfh}\wedge\hat{\omega}^{k-1}}$. Hence, $S_\theta^\pm$'s are the sum of $(k+2)$-times compositions of $\Lc^{m_{\Xfh}\wedge\omega_{\Xfh}^k}$, $\Lc^{\alpha_{\Xfh}\wedge\omega_{\Xfh}^{k-1}}$ or $\Lc^{dd^c u_\theta\wedge\omega_{\Xfh}^{k-1}}$. The terms containing only either $\Lc^{m_{\Xfh}\wedge\omega_{\Xfh}^k}$ or $\Lc^{\alpha_{\Xfh}\wedge\omega_{\Xfh}^{k-1}}$ are independent of $\theta$. So, the dominating part of the operator is $\left(\Lc^{dd^cu_\theta\wedge \omega_{\Xfh}^{k-1}}\right)^{k+2}$ and the previous lemma implies the desired estimate.
\end{proof}



\begin{lemma}\label{lem:mass_regular_transform}
	Let $S$ be a positive $(p, p)$-current. There exists a constant $c>0$ such that for all $\theta\in\C^*$ with $|\theta|\ll 1$, we have
	\begin{align*}
		\|\Lc_\theta^+(S)\|\le c(1+|\theta|)\|S\|,\quad \|\Lc^-(S)\|\le c\|S\|\quad\textrm{ and }\quad \|\Lc_\theta(S)\|_*\le c(2+|\theta|)\|S\|
	\end{align*}
	where the constant $c>0$ is independent of $S$ and $\theta$.
\end{lemma}

\begin{proof}
	From Lemma \ref{lem:reg_smooth_estimate}, we get
	\begin{align*}
		\langle \Lc_\theta^+(S), \omega^{k-p}\rangle=\langle \Lc_\theta(S) +\Lc^-(S), \omega^{k-p}\rangle=\langle S, \Lc_\theta(\omega^{k-p})+\Lc^-(\omega^{k-p})\rangle\le c'(1+ |\theta|)\|S\|
	\end{align*}
	and
	\begin{align*}
		\langle \Lc^-(S), \omega^{k-p}\rangle=\langle S, \Lc^-(\omega^{k-p})\rangle\le c''\|S\|.
	\end{align*}
	for some $c'>0$ and $c''>0$. Using the definition of the $*$-norm, we get the last inequality.
\end{proof}

\begin{lemma}\label{lem:disjoint_estimate_1}
	Let $S_1$ and $S_2$ are positive currents of bidegree $(p_1, p_1)$ and $(p_2, p_2)$ such that $p_1+p_2=k$ and $\supp S_1\cap \supp S_2=\emptyset$. Then, for all $\theta\in\C^*$ with $|\theta|\ll 1$, we have
	\begin{align*}
		\langle S_1, \Lc^\theta(S_2)\rangle = 0.
	\end{align*}
\end{lemma}

\begin{proof}
	\cite[Lemma 2.4.1]{DS10-1} implies that the support of $\pi_*\left((\alpha_{\Xfh}+dd^cu_\theta)\wedge\omega_{\Xfh}^{k-1}\right)$ uniformly shirinks to $\Delta$ as $\theta\to 0$.
\end{proof}

\begin{lemma}\label{lem:disjoint_estimate_2}
	Let $S_1$ and $S_2$ are positive currents of bidegree $(p_1, p_1)$ and $(p_2, p_2)$ such that $p_1+p_2=k$ and $\supp S_1\cap \supp S_2=\emptyset$. Then, we have
	\begin{align*}
		\left|\langle S_1, \Lc^-(S_2)\rangle\right| \lesssim \dist(\supp S_1, \supp S_2)^{2-2k}\|S_1\|\cdot\|S_2\|.
	\end{align*}
	The inequality $\lesssim$ means $\le$ up to a multiplicative constant independent of $S_1$ and $S_2$.
\end{lemma}

\begin{proof}
	The support of $\pi_1^*(S_1)\wedge \pi_2^*(S_2)$ is away from $\Delta$. On its support, \cite[Lemma 3.1]{DS04} implies that
	\begin{align*}
		\left\|\pi_*\left(\omega_\Xfh^k\right)\right\|_{\infty, \supp \pi_1^*(S_1)\wedge\pi_2^*(S_2)}\lesssim \dist(\supp S_1, \supp S_2)^{2-2k}.
	\end{align*}
	Since $S_1$ and $S_2$ are positive, the mass of $\pi_1^*(S_1)\wedge \pi_2^*(S_2)$ is bounded by $\|S_1\|\cdot\|S_2\|$.
\end{proof}

\section{Analytic (sub)multiplicative cocycles}\label{sec:analytic_multi_cocycle}
The notion of analytic (sub)multiplicative cocycles was first introduced by Favre \cite{FA2000}, \cite{FA2000-1} and further studied by Dinh \cite{Dinh09} and Gignac \cite{GIGNAC}.
\medskip

Let $Z$ be an irreducible compact complex space of dimension $l$, not necessarily smooth. Let $g: Z\to Z$
be an open holomorphic map.
\begin{definition}
	[Definition 1.1 in \cite{Dinh09}]\label{def:anal_multi_cocycle} A sequence $\{\kappa_n\}$ of functions $\kappa_n: Z\to (0, \infty)$ for $n\ge 0$ is said to be an analytic submultiplicative (resp., multiplicative) cocycle (with respect to $g$) if for all $m, n\ge 0$ and for all $z\in Z$,
	\begin{enumerate}
		\item $\kappa_n$ is upper-semicontinuous (usc for short) with respect to the Zariski topology on $Z$ and $\kappa_n\ge c_\kappa^n$ for some constant $c_\kappa>0$, and
		\item $\kappa_{m+n}(z)\le \kappa_n(z)\cdot \kappa_m(g^n(z))$ (resp., $=$).
	\end{enumerate}
\end{definition}

\begin{definition}
	[Introduction in \cite{Dinh09}] $$ \kappa_{-n}(z):=\max_{w\in g^{-n}(z)}\kappa_n(w).$$
\end{definition}

Observe that $\kappa_{-n}$ is usc in the Zariski sense. The following theorem is the key in this section.
\begin{theorem}[Theorem 1.2 in \cite{Dinh09}]\label{thm:limit_anal_multi_cocycle}
	The sequence $\{(\kappa_{-n})^{1/n}\}$ converges to a function $\kappa_-$ defined over $Z$ with the following properties: for all $\delta>\inf_Z \kappa_-$, the set $\{\kappa_-\ge \delta\}$ is a proper analytic subset of $Z$, invariant under $g$ and contained in the orbit of $\{\kappa_n\ge \delta^n\}$ for all $n\ge 0$. In particular, $\kappa_-$ is usc in the Zariski sense.
\end{theorem}
\medskip

The above notions and related properties work well on compact K\"ahler manifolds and so we will adopt the settings used in \cite{Ahn16}. From now on, we assume the hypotheses of Theorem \ref{thm:main}. Then, $f$ is a finite-to-one map and by the open mapping theorem between complex spaces, $f$ is open.
\subsection{Local multiplicity of $f^n$}\label{subsec:local_multi}
For each $n\in\N$, define $\mu_n(x)$ to be the local multiplicity of $f^n$ at $x\in X$. Then, $\{\mu_n\}$ is an analytic multiplicative cocycle with respect to $f$. By Theorem \ref{thm:limit_anal_multi_cocycle}, the limit function $\mu_-$ with $\min_X \mu_-=1$ exists for the sequence $\{\mu_n\}$. Since $X$ is compact and $\mu_-$ is usc in the Zariski sense, there exists a constant $d_f>0$ such that $d_f=\max_X\mu_-$. Since the topological degree of the map $f$ is $d_k$, $\mu_1(x)\le d_k$ for all $x\in X$ and therefore, by (2) of Definition \ref{def:anal_multi_cocycle}, $(\mu_{-n}(x))^{1/n}\le d_k$ and $d_f\le d_k$.
\medskip

If $d_f>1$, then we define $E_\lambda:=\{\mu_-\ge d_f\lambda^{-1}\}$ for $1\le\lambda<d_f$. Then, $f(E_\lambda)\subseteq E_\lambda$.

\subsection{Multiplicity of the analytic subset defined by the set of critical values of $f^n$}
For notational convenience, we denote $\xi(x, \varphi):=\nu(x, dd^c\log|\varphi|)$ for a point $x\in X$ and a holomorphic function $\varphi:X\to\C$ where $\nu(x, dd^c\log|\varphi|)$ denotes the Lelong number of the current $dd^c\log|\varphi|$ at $x\in X$. Then, by the chain rule, we have $\xi(x, J_{f^{m+n}})=\xi(x, J_{f^n})+\xi(x, J_{f^m}\circ f^n)$ for any $x\in X$ and for any $m, n\in \N$, where $J_{f^n}$ denotes the Jacobian determinant of $f^n$. We also have the following proposition:
\begin{proposition}
	[Remark 3 in \cite{Favre}; for a sharper version, see also \cite{Parra}] For any $x\in X$ and for any $m, n\ge 0$, the following inequality holds:
	\begin{align*}
		\xi(x, J_{f^m}\circ f^n)\le (2k-1+2\xi(x, J_{f^n}))\cdot \xi(f^n(x), J_{f^m}).
	\end{align*}
\end{proposition}

For each $n\in \N$, we define $\mu_n'(x):= 2k-1+2\xi(x, J_{f^n})$ on $X$. Then, the above proposition implies that $\{\mu_n'\}$ is an analytic submultiplicative cocycle with respect to $f$ (see \cite[Section 3]{FJ}). Hence, by Theorem \ref{thm:limit_anal_multi_cocycle}, the limit function $\mu_-'$ exists for $\{\mu_n'\}$. We have $\min_X\mu_-'=1$. Since $\mu_-'$ is usc in the Zariski sense, there exists a constant $d_f'>0$ such that $d_f'=\max_X\mu_-'$.
\medskip

If $d_f'>1$, we define $E'_{\lambda'}:=\{\mu'_-\ge d_f'(\lambda')^{-1}\}$ for $1\le\lambda'<d_f'$. 
%
%
Then, $f(E'_{\lambda'})\subseteq E'_{\lambda'}$.
\begin{lemma}
	[Lemma 2.6 in \cite{Ahn16}]\label{lem:choice_exceptional_set} Assume that $d_f, d_f'>1$. Let $E_\lambda$ and $E'_{\lambda'}$ be defined for $\lambda$ and $\lambda'$ with $1<\lambda<d_f$ and $1<\lambda'<d_f'$, respectively. Let $E:=E_\lambda\cup E'_{\lambda'}$. Then, $E$ is invariant under $f$ and there exists $n_E\in\N$ such that for every $m\in\N$, 
	\begin{enumerate}
		\item $\displaystyle \mu_{-n_Em}(x)<\left(\frac{d_f}{\lambda}\right)^{n_Em}$ for $x\in\Psi_{n_Em}\setminus E$;
		\item $\displaystyle\mathrm{multi}_x \Psi_{n_Em}=\nu(x,[\Psi_{n_Em}])<c_\Psi n_Em\left(\left(\frac{d_f}{\lambda}\right)^k\frac{d_f'}{\lambda'}\right)^{n_Em}$ for $x\in\Psi_{n_Em}\setminus E$,
	\end{enumerate}
	where $c_\Psi$ denotes the number of the irreducible components in the hypersurface $\Psi_1$ of the critical values of $f$. 
\end{lemma}

\section{H\"older Continuity}\label{sec:Holder}
In this section, we collect the H\"older continuity properties of some functions and some superpotentials which will be used to prove Theorem \ref{thm:main}.

\begin{definition}
	Let $K$ and $\alpha$ be positive constants. Let $U$ be an open subset of $X$. A continuous map $g:X\to \R$ is said to be $(K, \alpha)$-H\"older continuous on $U$ if for every $x, y \in U$, we have
	\begin{align*}
		|g(x)-g(y)|\le K\dist(x, y)^\alpha.
	\end{align*} 	
\end{definition}

In this section, we consider $X$ and $f:X\to X$ as in Theorem \ref{thm:main}. Let $E$ be an analytic subset invariant under $f$ and $\delta>1$ a real number such that $\mu_{-1}<\delta$ on $X\setminus E$, where $\mu_{-1}$ is as in Subsection \ref{subsec:local_multi}. 
As a corollary to \cite[Proposition 4.2]{DS10}, we obtain the following corollary:
\begin{corollary}[Corollary 4.4 in \cite{DS10}]\label{cor:lojasiewicz}
	There are an integer $N_E$ and a constant $c_E\geq 1$ such that if $0<t<1$ is a constant and if $x, y$ are two points in $X$ with $\dist(x, E)>t$ and $\dist(y, E)>t$, then we can write
	\begin{displaymath}
		f^{-1}(x)=\{x_1, ..., x_{d_k}\}\,\,\,\textrm{ and }\,\,\,f^{-1}(y)=\{y_1, ..., y_{d_k}\}
	\end{displaymath}
	with $\dist(x_i, y_i)\leq c_Et^{-N_E}\dist(x, y)^{1/\delta}$.
\end{corollary}

Let $\omega_f$ be a smooth closed $(1, 1)$-form in $\{f_*(\omega)\}$. Then, there exists a q-psh function $u_f:X\to\R$ such that $f_*(\omega)-\omega_f=dd^cu_f$. The function $u_f$ is unique up to a constant. By adding a proper constant, we may assume that $u_f<-1$. We can investigate the H\"older continuity property of $u_f$ outside the set $E$.
\begin{lemma}
	\label{lem:Holder_conti_qpotential} Assume that $0<s\ll 1$. Then, $u_f$ is $(c_fs^{-N_E}, \delta^{-1})$-H\"older continuous on $X \setminus E_s$ for some $c_f>0$ independent of $s$ where $N_E$ is a constant as in Corollary \ref{cor:lojasiewicz}.
\end{lemma}

\begin{lemma}\label{lem:smooth_approx_Holder}
	Let $g:X\to\R$ be a continuous function. Let $K$, $\alpha$ and $\theta$ be positive real numbers. Assume that $\theta\in\C^*$ be such that $|\theta|\ll 1$. Let $W_1$ and $W_2$ be two neighborhoods of $E$ such that $W_1^c\subset \overline{(W_1^c)_{\theta}}\subset W_{2}^c$. Assume that $g$ is $(K, \alpha)$-H\"older continuous on $W_{2}^c$. Then, there exists a smooth function $\widetilde g$ defined in a neighborhood of $W_1^c$ to $\R$ such that
	\begin{align*}
		\|\widetilde g\|_{C^2, W_1^c}\le c \theta^{-2(k+1)}\quad\textrm{ and }\quad \|g-\widetilde g\|_{\infty, W_1^c}\le c K\theta^{\alpha}
	\end{align*}
	where the constant $c>0$ is independent of $K$, $\alpha$, $\theta$, $W_1$ and $W_2$. 
\end{lemma}

\begin{proof}
	We first consider the following local case. Denote by $B$ and $B'$ two balls in a cooridnate chart with center at $0$ and of radii $1$ and $2$, respectively. Let $K$, $\alpha$ and $\theta$ be positive real numbers. Let $E\subset\C^k$ be an analytic subset, and $V$ and $W$ two open neighborhoods of $E$ such that $W^c\subset\overline{(W^c)_{\theta}}\subset V^c$. Assume that $B\setminus W\neq \emptyset$ in order to avoid triviality. Let $g:B'\to \R$ be a continuous function which is $(K, \alpha)$-H\"older continuous on $B'\setminus V$.
	\medskip
	
	Let $g_\reg$ denote the restriction to the set $B\setminus W$ of the standard regularization by convolution of $g$. Then, $g_\reg: B\setminus W\to\R$ becomes a desired smooth function. More precisely, let $\psi:\C^k\to\R$ be a smooth function such that $\supp\, \psi\Subset \{|z|<1\}$, $\psi(z)\ge 0$ for all $z\in \C^k$, $\psi(z)=\psi(|z|)$, $\psi$ is decreasing in $|z|$ and $\int_{\C^k}\psi(z)d\lambda(z)=1$ where $\lambda$ denotes the standard Lebesgue measure on $\C^k$. Then, $g_\reg$ is defined by
	\begin{align*}
		g_\reg(x):=\int_{y\in \C^k}g(x-y)|\theta|^{-2k}\psi(y/\theta)d\lambda(y).
	\end{align*}
	
	Since the support of $\psi(y/\theta)$ sits inside $\{|y|<|\theta|\}$, for any $x\in B-W$, 
	\begin{align*}
		|g(x)-g_\reg(x)|&=\left|\int_{y\in \C^k}(g(x)-g(x-y))|\theta|^{-2k}\psi(y/\theta)d\lambda(y)\right|\\
		&\le \int_{y\in \C^k}|g(x)-g(x-y)||\theta|^{-2k}\psi(y/\theta)d\lambda(y)\\
		&\le \int_{y\in \C^k}K|y|^\alpha|\theta|^{-2k}\psi(y/\theta)d\lambda(y)\le \int_{y\in \C^k}K|\theta|^\alpha|\theta|^{-2k}\psi(y/\theta)d\lambda(y)\le K|\theta|^\alpha
	\end{align*}
	For the second argument, we used a change of variables. This is also standard.
	\begin{align*}
		g_\reg(x)&=\int_{y\in \C^k}g(x-y)|\theta|^{-2k}\psi(y/\theta)d\lambda(y)=\int_{z\in \C^k}g(z)|\theta|^{-2k}\psi((x-z)/\theta)d\lambda(z).
	\end{align*}
	Since $\psi$ is smooth, we have
	$$\|g_\reg\|_{C^2}\lesssim \sup_{B'}|g|\cdot|\theta|^{-2(k+1)}\|\psi\|_{C^2}.$$
	Observe that the constants involved in the inequalites other than $K$, $\alpha$, $\theta$ and $\sup_{B'}|g|$ are some constants related to $\psi$. So, for a fixed $g:B'\to \R$, we get the same inequalities for different $K$, $\alpha$, $\theta$, $V$ and $W$ as long as $W^c\subset\overline{(W^c)_{\theta}}\subset V^c$ holds and $g$ is $(K, \alpha)$-H\"older continuous on $B'\setminus V$.
	\medskip
	
	Now, we consider the general case. At each point $x_0\in X$, we take an open neighborhood $U_{x_0}$ and a coordinate chart $\phi_{x_0}: U_{x_0}\to\C^k$ such that $B'\subset \phi_{x_0}(U_{x_0})$. This can be always achieved by scaling $\phi_{x_0}$. Then, we let $B_{x_0}:=\phi_{x_0}^{-1}(B)$ and $B'_{x_0}:=\phi_{x_0}^{-1}(B')$. Since $X$ is compact, we can find a finite cover $(B_j)_{j=1, \cdots, N}$ out of $(B_{x_0})_{x_0\in X}$. Let $(\chi_j)_{j=1, \cdots, N}$ be a partition of unity subordinated to the cover $(B_j)_{j=1, \cdots, N}$.
	\medskip
	
	Notice that on each $B_j'$, the Euclidean metric and the metric induced from $X$ are equivalent. So, there exists a constant $c_j>1$ such that $c_j^{-1}\dist(x, y)\le |x-y|\le c_j\dist(x, y)$ for $x, y\in B'_j$. So, we apply the above model case with $\theta$ replace by $c_j^{-1}\theta_i$ and with $V=W_{2}\cap B_j'$ and $W=W_1\cap B_j'$ and denote by $g_{i, j}$ the resulting function $g_\reg$ on $B_j$. Then, the desired function $g_i$ is obtained by
	\begin{align*}
		g_i=\sum_{j=1}^N \chi_j g_{i, j}.
	\end{align*}
\end{proof}

We also consider the H\"older continuity of some superpotentials. Since $f$ is not just a holomorphic correspondence but a holomorphic map, $f^*(\alpha)$ is smooth whenever $\alpha$ is smooth. Hence, we can apply the same argument as in the proof Lemma 5.4.3 in \cite{DS09} and we get the desired H\"older continuity.
\medskip


\begin{lemma}\label{lem:Holderconti_f_*}
	Let $p\in\{1, 2, \ldots, k\}$. The superpotential of $f_*(\omega^p)-\alpha_{f_*(\omega^p)}$ admits H\"older continuous superpotentials, where $\alpha_{f_*(\omega^p)}$ is a real closed smooth $(p, p)$-form in $\{f_*(\omega^p)\}$.
\end{lemma}

\begin{proof}
	We know that $f_*(\omega)$ admits H\"older continuous quasi-potentials. Then, by \cite[Proposition 3.4.2]{DS10-1}, we have the H\"older continuity of superpotentials of $[f_*(\omega)]^p$. Then, since $f_*(\omega^p)\le [f_*(\omega)]^p$, the domination principle in \cite[Theorem 1.1]{DNV} implies that the superpotential of $f_*(\omega^p)$ is H\"older continuous.
\end{proof}

\section{Proof of Theorem \ref{thm:main}}
From the assumption on $f$ we have $d_{p-1}<d_p$. We denote $d:=d_p/d_{p-1}$. We first determine an invariant set $E$ for $f$. Then, since $E$ is invariant under $f$, without loss of generality, it suffices to prove the theorem with $f$ replaced by some power $f^{N_f}$ of $f$. So, we find a good iterate $f^{N_f}$ of $f$ and prove the statement of Theorem \ref{thm:main} for $f^{N_f}$.\medskip

We will use $L:=d_p^{-1}f^*$ on $\Dc_p$ and $\Lambda:=d_{p-1}^{-1}f_*$ on $\Dc_{k-p+1}$. Note that since the set of indeterminacy is empty, $(f^n)^*=(f^*)^n$. So, $L^n(\cdot)=d_p^{-n}(f^n)^*(\cdot)$ on $\Dc_p$.
\medskip

Recall the two constants $d_f$ and $d_f'$ associated with $f$ as in Section \ref{sec:analytic_multi_cocycle}. We first consider the case of $d_f>1$ and $d_f'>1$ and the other cases later.

\subsection{$d_f>1$ and $d_f'>1$} We first set up some technical constants for the proof. We define a non-negative integer $n_\dyn$ such that 
$$d_f\le \left(\frac{d_s}{d_{s-1}}\right)^{n_\dyn}\quad\textrm{ for all }\,s=1, 2, \ldots, \pmain.$$
We can find a sufficiently large $N_1\in\N$ so that for every $n\ge N_1$, $$(40k^2c_\Psi n)^{8n_\dyn k}<d_f^{n},$$
where $c_\Psi$ is the number of the irreducible components in the hypersurface $\Psi_1$ of the critical values of $f$ as in Section \ref{sec:analytic_multi_cocycle}.
\medskip

Recall the definition of the dynamical degree. There exists a $N_2\in\N$ such that for every $n\ge N_2$,
\begin{align*}
	\langle (f^n)^*\omega^p, \omega^{k-p}\rangle\le \left(d_p+ \frac{1}{4}\right)^n\quad \textrm{ and }\quad\langle (f^n)_*\omega^{k-p+1}, \omega^{p-1}\rangle\le \left(d_{p-1}+ \frac{1}{4}\right)^n.
\end{align*}
Let $c_m>0$ be a constant such that for every integer $1\le l\le k$, $-c_m\|\varphi\|_\infty\omega^l\le\varphi\le c_m\|\varphi\|_\infty\omega^l$ holds for every smooth $(l, l)$-form $\varphi$. Then, if $n\ge N_2$, we have
\begin{align*}
	\langle L^n(S), \omega^{k-p}\rangle &= \langle L^n(\alpha_S), \omega^{k-p}\rangle\le c_m\left(\frac{5}{4}\right)^n\|\alpha_S\|_\infty
\end{align*}
for every current $S\in\Cc_p$ where $\alpha_S$ is a smooth closed form in $\{S\}$. In the same way, if $n\ge N_2$
\begin{align*}
	\langle \Lambda^n(R), \omega^{p-1}\rangle &= \langle \Lambda^n(\alpha_R), \omega^{p-1}\rangle\le  c_m\left(\frac{5}{4}\right)^n\|\alpha_R\|_\infty
\end{align*}
for every current $R\in\Cc_{k-p+1}$ where $\alpha_R$ is a smooth closed form in $\{R\}$.
\medskip

Then, take $N=N_1+N_2$ and we 
can choose $1<\lambda<d_f$ such that
\begin{align*}
	(40k^2c_\Psi N)^{(Nk)^{-1}}d_f^{1-(8n_\dyn k^2)^{-1}}<\lambda<d_f.
\end{align*}
Further, 
we can also choose $1<\lambda'<d_f'$ such that
\begin{align}
	\label{eq:multiplicity}1<\left(\frac{d_f}{\lambda}\right)^k\left(\frac{d_f'}{\lambda'}\right)<(40k^2c_\Psi N)^{-N^{-1}}d_f^{(8n_\dyn k)^{-1}}<(40k^2c_\Psi N)^{-N^{-1}}d^{(8 k)^{-1}}.
\end{align}
\medskip

For such $\lambda$ and $\lambda'$, we set
$$E:=E_\lambda\cup E_\lambda'$$
where $E_\lambda:=\{\mu_-\ge d_f\lambda^{-1}\}$ and $E'_{\lambda'}:=\{\mu'_-\ge d_f'(\lambda')^{-1}\}$ as in Section \ref{sec:analytic_multi_cocycle}. For each $j\in\N$ and for $\displaystyle \delta_j:=c_\Psi n_EjN\left(\left(\frac{d_f}{\lambda}\right)^k\frac{d_f'}{\lambda'}\right)^{n_EjN}$, Lemma \ref{lem:choice_exceptional_set} with $m=jN$ implies
\begin{enumerate}
	\item $\displaystyle \mu_{-n_EjN}(x)<\left(\frac{d_f}{\lambda}\right)^{n_EjN}<\delta_j$ for $x\in\Psi_{n_EjN}\setminus E$;
	\item $\displaystyle\mathrm{multi}_x \Psi_{n_EjN}=\nu(x,[\Psi_{n_EjN}])<c_\Psi n_EjN\left(\left(\frac{d_f}{\lambda}\right)^k\frac{d_f'}{\lambda'}\right)^{n_EjN}=\delta_j$ for $x\in\Psi_{n_EjN}\setminus E$.
\end{enumerate}

We look into the relationship between $\delta_j$ and $d$ which is a crucial condition for the proof. For each $j\in\N$, from \eqref{eq:multiplicity}
and $40k^2c_\Psi N>3$, we get
\begin{align*}
	40k^2\delta_j=(40k^2c_\Psi n_EjN)\left(\left(\frac{d_f}{\lambda}\right)^k\left(\frac{d_f'}{\lambda'}\right)\right)^{n_EjN}<(40k^2c_\Psi N)^{n_Ej}\left(\left(\frac{d_f}{\lambda}\right)^k\left(\frac{d_f'}{\lambda'}\right)\right)^{n_EjN}<d^{n_EjN(8 k)^{-1}}.
\end{align*}
\medskip

\textcolor{black}{
	Hence, we take $j=n_\dyn$ and $N_f=n_\dyn n_E N$. Since $E$ is invariant under $f$, it is invariant under $f^{N_f}$. We replace $f$ by $f^{N_f}$, $\lambda$ by $\lambda^{N_f}$ and $\lambda'$ by $(\lambda')^{N_f}$. Also $\Psi_1$ is replaced by $\Psi_{N_f}$, which will be denoted by $V$. The dynamical degree $d_s$ is replaced by $d_s^{N_f}$ for $s=0, 1, \ldots, k$, and $d$ by $d^{N_f}$. Also, $d_f$ and $d'_f$ are replaced by $d_f^{N_f}$ and ${d_f'}^{N_f}$, respectively. Then, we may assume that $f:X\to X$ satisfy
	\begin{align}
		\label{eq:ConditionM-1}&\mu_{-1}(x)<\delta\textrm{ for }x\in V\setminus E\quad\textrm{ and }\\
		\label{eq:ConditionM-2}&\mathrm{multi}_x V<\delta\textrm{ for }X\in V\setminus E,
	\end{align}
	where $\displaystyle \delta:=c_\Psi n_\dyn n_EN\left(\frac{d_f}{\lambda}\right)^k\left(\frac{d_f'}{\lambda'}\right).$
} Note that in the expression of $\delta$, we keep $c_\Psi$ from the replacement, and that $n_\dyn$, $n_E$ and $N$ remain unchanged. The number $c_\Psi n_\dyn n_EN$ plays the role of an upper bound of the number of irreducible components in $\Psi_{{N_f}}$. After the replacement, we have
\begin{align}
	\label{cond:delta_condition}(40k^2\delta)^{6k}<d^{3/4}
\end{align}
and mass estimates as follows:
\begin{align}
	\label{cond:mass_condition}\|L^n(S)\|\le c_m\left(\frac{5}{4}\right)^{N_fn}\|\alpha_S\|_\infty\quad\textrm{ and }\quad \|\Lambda^n(R)\|\le c_m\left(\frac{5}{4}\right)^{N_fn}\|\alpha_R\|_\infty
\end{align}
for $S\in\Cc_p$ and for $R\in\Cc_{k-p+1}$.\smallskip

The following two lemmas will be used to determine appropriate neighborhoods of the set $E$ and those of subsets of $V\setminus E$. Lemma \ref{lem:EV_dist} was induced from a Lojasiewicz type inequality as in \cite{Ahn16}. See also \cite{FS-II} and \cite{DS10}. It is of local nature, and so it is valid in our case as well.\smallskip

Let $\alpha_V$ be a closed smooth $(1, 1)$-form in $\{[V]\}$. Then, we can find a unique q-psh function $\varphi_V$ over $X$ such that $\sup_X \varphi_V=0$ and $dd^c\varphi_V=[V]-\alpha_V$. Let $E_V=V\cap E$. Then, from our construction, for all $x\in V\setminus E_V$, $\mathrm{multi}_xV<\delta$ from \eqref{eq:ConditionM-2}.
\begin{lemma}[Lemma 3.3 in \cite{Ahn16}]\label{lem:EV_dist}
	There are constants $C, A>0$ such that for $x\in X$,
	\begin{align*}
		\delta\log\dist(x, V) + C\log\dist(x, E_V)-A\le \varphi_V(x)\le \log\dist(x, V)+A.
	\end{align*}
\end{lemma}

\begin{lemma}
	[Lemma 3.1 in \cite{DS10}]\label{lem:smooth_region} There is a constant $m_f\ge 1$ such that for every subset $A$, $B$ of $X$,
	\begin{align*}
		\dist(f^{-j}(A), f^{-j}(B))\ge m_f^{-j}\dist(A, B)\quad \textrm{ for }j= 0, 1, 2, \ldots.
	\end{align*}
	In particular, if $f(B)\subseteq B$, 
	\begin{align*}
		\dist(f^{-j}(A), B)\ge m_f^{-j}\dist(A, B)\quad \textrm{ for }j= 0, 1, 2, \ldots.
	\end{align*}
\end{lemma}
Let $\varepsilon>0$ be a sufficiently small positive real number so that $\varepsilon<m_f^{-1}$. Later, this number will be precisely determined in terms of $n$. We define the following sequences of real numbers and currents for $n\in\N$ and $1\le i\le n$.
\begin{itemize}
	\item $s_{n,i}=\varepsilon^{nCi}$,
	
	\item $\varepsilon_{n,i}=\varepsilon^{nC(4k+N_E)(40k^2\delta)^{6ki}}$, and
	
	\item $t_{n,i}={\varepsilon_{n,i}}^{(10k)^{-1}}$,
	
	\item $R_{n, 0}=R$ and $R_{n, i}:=(\Lambda(R_{n, i-1}))_{\varepsilon_{n, i}}$.
\end{itemize}
where the subscript means the regularization \eqref{eq:theta_regularization} as in Section \ref{sec:regularization}. The constant $C>0$ is as in Lemma \ref{lem:EV_dist} and the constant $N_E$ is as in Corollary \ref{cor:lojasiewicz}.\\

Now, we prove the statement of Theorem \ref{thm:main} with the new $f$ obtained from the replacement. Let $S\in\Cc_p$ be such that $S$ has a smooth representation in a neighborhood of $E$. Let $\varphi$ be a smooth test $(k-p, k-p)$-form. Our goal is to show
\begin{align*}
	|\langle L^n(S-\alpha_S), \varphi\rangle|=\left| \Uc_{L^n(S-\alpha_S)}(dd^c\varphi)\right|\le c\|\varphi\|_{C^2}\rho^n\quad\textrm{ for all }n\in\N
\end{align*}
for some constants $c>0$ and $0<\rho<1$ independent of $\varphi$. So, it suffices to show that there exist $c'>0$ and $0<\rho<1$ such that $$\left|\Uc_{L^n(S-\alpha_S)}(R)\right|\le c'\rho^n$$ when $n\gg 1$ for smooth $R\in\Dc^0_{k-p+1}$ such that $R=R_+-R_-$, where $R_\pm$ are smooth currents in $\Cc_{k-p+1}$ with $\|R_\pm\|_\infty\le 1$. Then, following \cite{DS09} and \cite{Ahn16}, we can expand
\begin{align*}
	\langle U_{L^n(S-\alpha_S)}, R\rangle&=\langle LU_{L^{n-1}(S-\alpha_S)}, R\rangle=d^{-1}\langle U_{L^{n-1}(S-\alpha_S)}, \Lambda(R)\rangle\\
	&=d^{-1}\langle U_{L^{n-1}(S-\alpha_S)}, \Lambda(R_{n,0})-R_{n,1}\rangle+d^{-1}\langle U_{L^{n-1}(S-\alpha_S)},R_{n,1}\rangle\\
	&=d^{-1}\langle U_{L^{n-1}(S-\alpha_S)}, \Lambda(R_{n,0})-R_{n,1}\rangle+d^{-2}\langle U_{L^{n-2}(S-\alpha_S)},\Lambda(R_{n,1})\rangle\\
	&=d^{-1}\langle U_{L^{n-1}(S-\alpha_S)}, \Lambda(R_{n,0})-R_{n,1}\rangle\\
	&\quad\quad\quad + \cdots + d^{-i}\langle U_{L^{n-i}(S-\alpha_S)},\Lambda(R_{n,i-1})-R_{n,i}\rangle\\
	&\quad\quad\quad + \cdots + d^{-n}\langle U_{S-\alpha_S},\Lambda(R_{n,n-1})-R_{n,n}\rangle +d^{-n}\langle U_{S-\alpha_S},R_{n,n}\rangle
\end{align*}
where the potentials $U_{L^{n-i}(S-\alpha_S)}$'s are the Green potentials. Here, the pairings in the above make sense since $R_{n,i}$'s have $C^1$-coefficients and so, $\Lambda(R_{n,i})$'s admit continuous superpotentials for all $n\in\N$ and for all $1\le i\le n$. Since the operator $\Lambda$ and the regularizing operators are linear, it is enough to estimate
\begin{enumerate}
	\item $d^{-i}\langle U_{L^{n-i}(S-\alpha_S)},\Lambda(R_{n,i-1})-R_{n,i}\rangle$ for smooth $R\in\Cc_{k-p+1}$ with $\|R\|_\infty\le 1$ and for $i=1, \cdots, n$, and
	\item $d^{-n}\langle U_{S-\alpha_S},R_{n,n}\rangle$ for smooth $R\in\Dc^0_{k-p+1}$ such that $R=R_+-R_-$ where $R_\pm$ are smooth currents in $\Cc_{k-p+1}$ with $\|R_\pm\|_\infty\le 1$.
\end{enumerate}

Suppose that when $\varepsilon>0$ is sufficiently small, we have (1) $\lesssim \varepsilon$ and (2) $\lesssim nd^{n/4}(-\log\varepsilon)$ for all integers $n$ and $i$ with $1\le i\le n$, where the inequality $\lesssim$ means $\le$ up to a multiplicative constant independent of $n$ and $i$. (These estimates of (1) and (2) will be proved in Sections \ref{sec:(1)} and \ref{sec:(2)}.) Then, from the expansion, we have
\begin{align*}
	\left|\langle U_{L^n(S-\alpha_S)}, R\rangle\right|&\lesssim 6n\varepsilon+nd^{n/4}(-\log\varepsilon)
\end{align*}
when $0<\varepsilon\ll 1$. We take $\varepsilon=d^{-n}$. Then, $\rho=d^{1/8}$ proves the statement for the case of $d_f>1$ and $d_f'>1$.
\medskip

\subsection{$d_f=1$ or $d_f'=1$} In these cases, we get either $E=\emptyset$ or $E'=\emptyset$. We will consider only the case of $d_f=1$ and $d_f'=1$. The other two cases can be deduced in the same way. It is enough to find a good iterate $f^{N_f}$ of $f$ so that we can apply the previous arguments. Recall $d:=d_p/d_{p-1}$. Choose a sufficiently large $N_1\in\N$ so that for every $n\ge N_1$, we have
\begin{align*}
	(40k^2c_\psi n)^{8k}<d^n.
\end{align*}
Then, as in the case of $d_f>1$ and $d_f'>1$, we can find $1<\lambda<d$ and $0<\lambda'<1$ so that
\begin{align*}
	1<\left(\frac{d}{\lambda}\right)^k\left(\frac{1}{\lambda'}\right)<(40k^2c_\psi N_1)^{-N_1^{-1}}d^{(8k)^{-1}}.
\end{align*}
As in the case of $d_f>1$ and $d'_f>1$, there exist constants $c_m>0$ and $N_2\in\N$ such that
\begin{align*}
	\langle L^n(S), \omega^{k-p}\rangle \le c_m\left(\frac{5}{4}\right)^n\|\alpha_S\|_\infty\quad\textrm { and }\quad\langle \Lambda^n(R), \omega^{p-1}\rangle \le c_m\left(\frac{5}{4}\right)^n\|\alpha_R\|_\infty
\end{align*}
for every $n\ge N_2$ where $S\in\Cc_p$ and $R\in\Cc_{k-p+1}$ are currents, and $\alpha_S$ and $\alpha_R$ are smooth closed forms in $\{S\}$ and $\{R\}$, respectively. Let $N=N_1+N_2$.
\medskip

For such $1<\lambda<d$ and $0<\lambda'<1$, as remarked in the above, we have $E:=\emptyset$ and the arguments in Lemma \ref{lem:choice_exceptional_set} imply that for some large enough $n_E\in\N$, we have 
\begin{enumerate}
	\item $\displaystyle \mu_{-n_EN}(x)<\left(\frac{d}{\lambda}\right)^{n_EN}<\delta$ for $x\in\Psi_{n_EN}$;
	\item $\displaystyle\mathrm{multi}_x \Psi_{n_EN}<c_\Psi n_EN\left(\left(\frac{d}{\lambda}\right)^k\frac{1}{\lambda'}\right)^{n_EN}=\delta$ for $x\in\Psi_{n_EN}$,
\end{enumerate}
where $\displaystyle \delta=c_\Psi n_EN\left(\left(\frac{d}{\lambda}\right)^k\frac{1}{\lambda'}\right)^{n_EN}$ and $c_\Psi$ denotes the number of the irreducible components in the hypersurface $\Psi_1$ of the critical values of $f$. We replace $f$ by $f^{N_f}$ where $N_f:=n_EN$. Then, with these multiplicity conditions, our arguments in the case of $d_f>1$ and $d'_f>1$ work in the same way and give us the desired conclusion. Note that in this case, we have $E=\emptyset$, we only need to compute on $W_{n,i,1}$ and $W_{n,i,3}$ (for these sets, see Section \ref{sec:(1)}) and that every positive closed current of bidegree $(p,p)$ satisfies Theorem \ref{thm:main}.
\hfill $\square$
\bigskip

We end this section with a lemma about some estimates related to the currents $R_{n,i}$, which were not needed in the case of $\P^k$ due to the simplicity of its cohomology.

\begin{lemma}
	\label{lem:mass_estimates}
	Let $C_m=2c_\reg c_m\left(\frac{5}{4}\right)^{N_f}$ where $c_\reg$ is as in Proposition \ref{prop:C1_estimate_regularization} and $c_m$ as in \eqref{cond:mass_condition}. Let $R\in\Cc_{k-p+1}$ be a smooth current. Then, for all positive integers $n$ and $i$ such that $1\le i\le n$, the current $R_{n,i}$ can be expressed $R_{n,i}=R_{n,i}^+-R_{n,i}^-$ where $R_{n,i}^\pm$ are positive closed $(k-p+1, k-p+1)$-currents with $C^1$-coefficients such that
	\begin{enumerate}
		\item $\displaystyle \|R_{n,i}^+\|_{C^1}+\|R_{n,i}^-\|_{C^1}\le \|R\|_{C^1} (C_m)^i\left(\prod_{j=1}^i\varepsilon_{n,j}\right)^{-6k^2}$ and
		\item $\displaystyle \|R_{n,i}^+\|+\|R_{n,i}^-\|\le  \|R\|_{C^1}(C_m)^i \left(\prod_{j=1}^{i-1}\varepsilon_{n,j}\right)^{-6k^2}$.
	\end{enumerate}
\end{lemma}

\begin{proof}
	We prove by induction. When $i=1$, Then, $R_{n,1}=(\Lambda(R))_{\varepsilon_{n, 1}}^+-(\Lambda(R))_{\varepsilon_{n, 1}}^+$. Proposition \ref{prop:C1_estimate_regularization} and the estimates \eqref{cond:mass_condition} prove it. So, we may assume that (1) and (2) are true for $i-1$ with $i\ge 2$. Denote by $r_{n,i}=\|R_{n,i}^+\|_{C^1}+\|R_{n,i}^-\|_{C^1}$. Then, from Section \ref{sec:regularization}, we have
	\begin{align*}
		R_{n,i}=(\Lambda(R_{n,i-1}^+))^+_{\varepsilon_{n, i}}+(\Lambda(R_{n,i-1}^-))^-_{\varepsilon_{n, i}}-(\Lambda(R_{n,i-1}^+))^-_{\varepsilon_{n, i}}-(\Lambda(R_{n,i-1}^-))^+_{\varepsilon_{n, i}}
	\end{align*}
	and define
	\begin{align*}
		R_{n,i}^+:=(\Lambda(R_{n,i-1}^+))^+_{\varepsilon_{n, i}}+(\Lambda(R_{n,i-1}^-))^-_{\varepsilon_{n, i}}\quad\textrm{ and }\quad R_{n,i}^-:=(\Lambda(R_{n,i-1}^+))^-_{\varepsilon_{n, i}}+(\Lambda(R_{n,i-1}^-))^+_{\varepsilon_{n, i}}.
	\end{align*}
	From Proposition \ref{prop:C1_estimate_regularization} and \eqref{cond:mass_condition}, we have
	\begin{align*}
		\left\|\left(\Lambda(R_{n,i-1}^\pm)\right)^\pm_{\varepsilon_{n,i}}\right\|_{C^1}\le c_\reg\varepsilon_{n,i}^{-6k^2}\left\|\Lambda(R_{n,i-1}^\pm)\right\|\le c_\reg c_m\left(\frac{5}{4}\right)^{N_f}\varepsilon_{n,i}^{-6k^2}\|R_{n,i-1}^\pm\|_{C^1},
	\end{align*}
	respectively and 
	thus by the triangle inequality, we get
	\begin{align*}
		r_{n,i}\le 2c_\reg c_m\left(\frac{5}{4}\right)^{N_f}\varepsilon_{n,i}^{-6k^2}r_{n,i-1}.
	\end{align*}
	
	For the mass estimate, again from 	\eqref{cond:mass_condition}, we have
	\begin{align*}
		\|R_{n,i}^+\|&\le \|(\Lambda(R_{n,i-1}^+))^+_{\varepsilon_{n,i}}\|+\|(\Lambda(R_{n,i-1}^-))^-_{\varepsilon_{n,i}}\|\le c_\reg(\|\Lambda(R_{n,i-1}^+)\|+\|\Lambda(R_{n,i-1}^-)\|)\\
		&\le c_\reg c_m\left(\frac{5}{4}\right)^{N_f}r_{n,i-1}.
	\end{align*}
	The same is true for $R_{n,i}^-$. So, we have proved (1) and (2) for $i$ as desired.
\end{proof}

\section{Estimates of $d^{-i}\langle U_{L^{n-i}(S-\alpha_S)},\Lambda(R_{n,i-1})-R_{n,i}\rangle$}\label{sec:(1)}
We start by constructing two families of cut-off functions playing the role of smoothing out the boundary of the neighborhoods of $E$ and a subset of $V\setminus E$, on each of which the currents $L^{n-i}(S-\alpha_S)$ and $f_*(\omega^{k-p+1})$ have different analytic properties.
\medskip

The following is a basic regularization lemma. We can get it through patching locally regularized functions on each coordinate chart by use of a partition of unity.
\begin{lemma}\label{lem:basic_cutoff}
	Let $E$ be the analytic subset as above. There are constants $s_0>0$ and  $0<r_E<1$ such that for all $0<s\le s_0$, there exists a smooth function $\chi^E_{s}:X\to\R$ such that $\chi^E_{s}\equiv 1$ on $E_{r_Es}$, $\supp \chi^E_{s}\Subset E_{s}$ and $\|\chi^E_{s}\|_{C^2}\lesssim 1/s^2$ where the inequality $\lesssim$ means $\le$ up to a multiplicative constant independent of $s$.
\end{lemma}
%
%
\textcolor{black}{
	The following crucial lemma is induced from Lemma \ref{lem:EV_dist}. The argument in \cite[Lemma 7.3]{Ahn16} is of local nature and so it works in our case as well. This is one of the two places where the multiplicity assumption comes to play. The other is the H\"older continuity of $u_f$ where $u_f$ is a function such that $f_*(\omega)-\omega_f=dd^c u_f$ as in Section \ref{sec:Holder}.}
\begin{lemma}[Lemma 7.3 in \cite{Ahn16}]\label{lem:cut-off}
	Let $n$ and $i$ be integers such that $1\le i\le n$. Let $s, t$ be two positive real numbers such that $\frac{1}{2}r_E^{k+2}s_{n,i}\leq s\leq 2s_{n,i}$ and $\frac{1}{2}t_{n,i}\leq t\leq 2{t_{n,i}}^{\left(\frac{2}{2\delta+1}\right)^{k+2}}$ for $s_{n,i}, t_{n,i}$ with $0<\varepsilon\ll 1$ and $r_E$ as in Lemma \ref{lem:basic_cutoff}. Then, there is a function $\chi_{s,t}:X\to\R$ with $0\leq\chi_{s,t}\leq 1$, such that $\chi_{s,t}\equiv 1$ on $V_t\setminus {E_s}$, $\supp (\chi_{s,t})\subseteq V_{t^{(\delta+1/2)^{-1}}}\setminus E_{r_Es}$, and $\|\chi_{s,t}\|_\dsh\leq c_\chi\max\left\{1, 9s^{-2}\right\}$, where $c_\chi>0$ is a constant independent of $s$, $t$, and $\varepsilon$, $i$, $n$ in the definitions of $s_{n,i}$, $t_{n,i}$.
\end{lemma}

We define 
\begin{align*}
	\chi_{n,i,1}&:=\left(1-\chi^E_{s_{n,i}/r_E}\right)\chi_{s_{n,i},t_{n,i}}\quad\textrm{ on }W_{n,i,1}:=V_{t_{n,i}^{(\delta+1/2)^{-1}}}\setminus E_{s_{n,i}},\\
	\chi_{n,i,2}&:=\chi^E_{s_{n,i}/r_E} \quad\textrm{ on } W_{n,i,2}:=E_{s_{n,i}/r_E}\quad\textrm{ and }\\
	\chi_{n,i,3}&:=1-\chi_{n,i,1}-\chi_{n,i,2}=\left(1-\chi^E_{s_{n,i}/r_E}\right)\left(1-\chi_{s_{n,i},t_{n,i}}\right) \quad \textrm{ on } W_{n,i,3}:=X\setminus \left(V_{t_{n,i}}\cup E_s\right).
\end{align*}
Then, we can write
\begin{align*}
	&d^{-i}\langle U_{L^{n-i}(S-\alpha_S)},\Lambda(R_{n,i-1})-R_{n,i}\rangle=\sum_{j=1}^3d^{-i}\langle \chi_{n,i,j}U_{L^{n-i}(S-\alpha_S)},\Lambda(R_{n,i-1})-R_{n,i}\rangle
\end{align*}
each of which will be estimated below.

\subsection{Estimates of $d^{-i}\langle \chi_{n,i,1}U_{L^{n-i}(S-\alpha_S)},\Lambda(R_{n,i-1})-R_{n,i}\rangle$ in $W_{n,i,1}$}
In this region, we use the H\"older continuity of the quasi-potential $u_f$ of $f_*(\omega)$. For $u_f$, see Section \ref{sec:Holder}.
\medskip


We will use the following notations throughout this subsection. Let $\theta\in\C^*$ be a constant. Later, $\theta$ will be chosen to be $\varepsilon_{n, i}$ with sufficiently small $\varepsilon$. Let $\Theta_j$ be a smooth positive closed $(k, k)$-form on $\Xfh$, which equals either $(m_{\widehat{\Delta}}\omega_{\Xfh}+dd^cu_\theta)\wedge \omega_{\Xfh}^{k-1}$ or $(m_{\widehat{\Delta}}\omega_{\Xfh}-\alpha_{\widehat{\Delta}})\wedge\omega_\Xfh^{k-1}$ for $j=1, \cdots, k+2$ so that its associated positive closed semi-regular transform $\Lc^{\Theta_j}$ of bidegree $(0, 0)$ coincides with either $\Lc_\theta^+$ or $\Lc^-$ for $j=1, \cdots, k+2$. For the operators $\Lc_\theta^+$'s and $\Lc^-$'s, see Section \ref{sec:regularization}. We also denote by $\Lc^m=\Lc^{\Theta_1}\circ \cdots\circ \Lc^{\Theta_m}$ for $0\le m\le k+2$. Notice that $\Lc^m$ is positive and closed, and that it may depend on $\Theta_j$'s and $\theta$.
\medskip

Differently from \cite{Ahn16}, we cannot use the negativity of $U_{L^{n-i}(S-\alpha_S)}$. See \cite{BGS}. So, we are estimating not just integrals of the form $\langle \chi_{n,i,1}\Lc^{\omega_\Xfh^{k-1}}(S), f_*(\omega^{k-p+1})\rangle$ and $\langle \chi_{n,i,1}\Lc^{u{\omega_\Xfh^{k-1}}}(S),$ $ f_*(\omega^{k-p+1})\rangle$, but more generally
those of the form $\langle \chi_{n,i,1}\Lc^{\omega_\Xfh^{k-1}}(S), \Lc^m(f_*(\omega^{k-p+1}))\rangle$ and $\langle \chi_{n,i,1}\Lc^{u{\omega_\Xfh^{k-1}}}(S),$ $ \Lc^m(f_*(\omega^{k-p+1}))\rangle$. By the continuity of the (semi-)regular transformation and the superpotential of $f_*(\omega^{k-p+1})$, we may assume that $S$ is smooth positive and closed due to \cite{DS04}.\medskip

\subsubsection{Estimate of $\langle \chi_{n,i,1}\Lc^{\omega_\Xfh^{k-1}}(S), \Lc^m(f_*(\omega^{k-p+1}))\rangle$}
\begin{lemma}\label{lem:tubular_closed_basic} Let $S\in\Cc_p$ be a smooth current. For $0<t\ll 1$, we have
	\begin{align*}
		\int_{V_t} \Lc^{{\omega_\Xfh^{k-1}}}(S)\wedge \Lc^m\left(\omega^{k-p+1}\right)\lesssim \|S\| t^{3/2}.
	\end{align*}
	The inequality $\lesssim$ means $\le$ up to a multiplicative constant independent of $t$, $\theta$, and $S$.
\end{lemma}

The proof is similar to the proof of \cite[Lemma 2.3.7]{DS09}. But $X$ is not homogeneous.

\begin{proof}
	Here since $S$, $\omega^{k-p+1}$ and ${\omega_\Xfh^{k-1}}$ are smooth and positive, we can write
	\begin{align*}
		&\int_{V_t} \Lc^{\omega_\Xfh^{k-1}}(S)\wedge \Lc^m\left(\omega^{k-p+1}\right)\le\left\|{\omega_\Xfh^{k-1}}\right\|_\infty\left\|\Lc^m\left(\omega^{k-p+1}\right)\right\|_\infty\int_{V_t}\Lc^{\omega_\Xfh^{k-1}}(S)\wedge \omega^{k-p+1}\\
		&\quad=\left\|\Lc^m\left(\omega^{k-p+1}\right)\right\|_\infty \left\|{\omega_\Xfh^{k-1}}\right\|_\infty \int_{y\in {V_t}}\left[\int_{x\in X\setminus\{y\}}\pi_*\left(\omega_\Xfh^{k-1}\right)(x, y)\wedge S(x)\right]\wedge \omega^{k-p+1}(y)
	\end{align*}
	
	Since $X$ is compact, it is enough to estimate in a coordinate neighborhood $B\subset X$ with compact closure. So, we prove that
	\begin{align*}
		\int_{y\in {V_t}\cap B}\left[\int_{x\in X\setminus \{y\}}\pi_*\left(\omega_\Xfh^{k-1}\right)(x, y)\wedge S(x)\right]\wedge \omega^{k-p+1}(y)\lesssim t.
	\end{align*}
	
	%
	%
	Let $B_z(4t)$ denote the ball of radius $4t$ centered at some $z\in V\cap B$. By shrinking $B$ if necessary, we may assume that a neighborhood of $\overline{B}$ sits inside the coordinate neighborhood and that $B_z(4t)$ lies in the coordinate chart for $0<t\ll 1$. We estimate
	\begin{align*}
		\int_{y\in B_z(4t)}\left[\int_{x\in X\setminus \{y\}}\pi_*\left(\omega_\Xfh^{k-1}\right)(x, y)\wedge S(x)\right]\wedge \omega^{k-p+1}(y).
	\end{align*}
	By the representation of a flat current in terms of a vector field and a Radon measure (or one might think of it as a version of the Radon-Nikodym theorem), we may write it as
	\begin{align*}
		\int_{x\in X}\left(\int_{y\in B_z(4t)\setminus\{x\}} \left[\Psi_S(x)\righthalfcup\pi_*
		\left(\omega_\Xfh^{k-1}\right)(x, y) \right]\wedge \omega^{k-p+1}(y)\right)d\mu(x).
	\end{align*}
	Here, $\Psi_S(x)$ is actually a $(k-p)$-vector field with $\|\Psi_S(x)\|_\infty\le 1$ and $\left[\Psi_S(x)\righthalfcup\pi_*\left(\omega_\Xfh^{k-1}\right)(x, y)\right]$ is actually a form in $y$ smooth except $x$ and the singularity is at most $\dist(x, y)^{2-2k}$ and one might think of $\mu$ as the trace measure of $S$. Hence, $\int_X d\mu$ can be bounded in terms of $\|S\|$. Since $\|\Psi_S(x)\|_\infty\le 1$ and $\omega_\Xfh^{k-1}$ is smooth, the inner integral is bounded by a uniform constant multiple of $\left[(\sum_{|I|=k-p} e_I(x))\righthalfcup\pi_*\left(\omega_{\widehat{\mathfrak{X}}}^{k-1}\right)(x, y)\right]\wedge \omega^{k-p+1}(y)$ (in Federer's notation) and further bounded by the following uniformly with respect to $y$.
	\begin{align*}
		\int_{B_z(4t)}|x|^{2-2k}\left(dd^c|x|^2\right)^k\lesssim t^2.
	\end{align*}
	So, the arguments using the Lelong number and the Abel transform as in \cite[Lemma 2.3.7]{DS09} complete the proof.
\end{proof}

Next, we modify \cite[Proposition 7.6]{Ahn16} to prove the following:
\begin{lemma}\label{lem:estimate_off_E-1}
	Let $n$ and $i$ be integers such that $1\le i\le n$. Let $l\in \{0, 1, \ldots, k-p+1\}$. Let $s, t$ be two positive real numbers such that $\frac{1}{2}r_E^{k-p-l+2}s_{n,i}\leq s\leq 2s_{n,i}$ and $\frac{1}{2}t_{n,i}\leq t\leq 2(t_{n,i})^{(\frac{2}{2\delta+1})^{k-p-l+2}}$ with $0<\varepsilon\ll 1$. Then, we have
	\begin{align*}
		&\left|\int_{V_t\setminus {E_s}}\Lc^{{\omega_\Xfh^{k-1}}}(S)\wedge \Lc^m\big(f_*(\omega)^l\wedge\omega^{k-p-l+1}\big)\right|\lesssim  \|\alpha_S\|_\infty s^{-2(k+l)-N_E}t^{\beta_l},
	\end{align*}
	where $\beta_l=(2(k+1)(2\delta+1)\delta)^{-l}$. The inequality $\lesssim$ means $\le$ up to a multiplicative constant independent of $n$, $i$, $s$, $t$, $\theta$, $S$ and $\alpha_S$.
\end{lemma}

For this estimate, we need an auxiliary lemma. Lemma \ref{lem:Holder_conti_qpotential} and \eqref{eq:ConditionM-1} imply that $u_f$ is H\"older continuous on $X$ and that $u_f$ is $(c_fs^{-N_E}, \delta^{-1})$-H\"older continuous in $X\setminus E_s$ for $s>0$. We write $$\chi^{E,j}:=\chi^E_{\left(\frac{1}{2}r_E\right)^js}$$ for $j=1, \cdots, k+2$, where $\chi^E_{\left(\frac{1}{2}r_E\right)^js}$ is as in Lemma \ref{lem:basic_cutoff}. Also, recall that the transforms $\Lc^{\Theta_j}$'s are all invariant under $(x, y)\to (y, x)$.
\begin{lemma}\label{lem:integral_semi-regular_tr} Under the assumptions of Lemma \ref{lem:estimate_off_E-1}, for any $1\le j \le m$, we have
	\begin{align*}
		&\Big|\int_X \chi_{s, t} \Lc^{\omega_\Xfh^{k-1}}(S)\wedge \Lc^{\Theta_1}\big(\cdots\Lc^{\Theta_{j-2}}\big(\Lc^{\Theta_{j-1}}\big(\big(1-\chi^{E,j-1}\big)\Lc^{\Theta_{j}}\big(\chi^{E, {j}}\Lc^{\Theta_{j+1}}\big(\chi^{E, {j+1}}\cdots\\
		&\quad\quad\quad\quad\quad\quad\chi^{E, {m-1}}\Lc^{\Theta_m}\big(f_*(\omega)^{l-1}\wedge dd^c \big(\chi^{E, {k+2}} u_f\big)\wedge \omega^{k-p-l+1}\big)\cdots\big)\big)\big)\big)\cdots\big)\Big|\lesssim \|S\|s^{-2k}t^{3/(2\delta+1)}
	\end{align*}
	where $\chi_{s, t}$ is the cut-off function as in Lemma \ref{lem:cut-off}.	The inequality $\lesssim$ means $\le$ up to a multiplicative constant independent of $n$, $i$, $s$, $t$, $\theta$ and $S$.
\end{lemma}

\begin{proof}
	\begin{align*}
		&\int_X \chi_{s, t} \Lc^{\omega_\Xfh^{k-1}}(S)\wedge \Lc^{\Theta_1}\big(\cdots\Lc^{\Theta_{j-2}}\big(\Lc^{\Theta_{j-1}}\big(\left(1-\chi^{E, {j-1}}\right)\Lc^{\Theta_{j}}\big(\chi^{E, {j}}\Lc^{\Theta_{j+1}}\big(\cdots\\
		&\quad\quad\quad\quad\quad\quad\quad\chi^{E, {m-1}}\Lc^{\Theta_m}\left(f_*(\omega)^{l-1}\wedge dd^c \left(\chi^{E, k+2} u_f\right)\wedge \omega^{k-p-l+1}\right)\cdots\big)\big)\big)\big)\cdots\big)\\
		&=\int_X \left(1-\chi^{E, {j-1}}\right)\Lc^{\Theta_{j-1}}\big(\Lc^{\Theta_{j-2}}\big(\cdots\Lc^{\Theta_1}\left(\chi_{s, t} \Lc^{\omega_\Xfh^{k-1}}(S)\right)\cdots\big)\big)\wedge \Lc^{\Theta_j}\big(\chi^{E, {j}}\Lc^{\Theta_{j+1}}\big(\cdots\\
		&\quad\quad\quad\quad\quad\quad\quad\chi^{E, {m-1}}\Lc^{\Theta_m}\left(f_*(\omega)^{l-1}\wedge dd^c \left(\chi^{E, k+2} u_f\right)\wedge \omega^{k-p-l+1}\right)\cdots\big)\big)
	\end{align*}
	The currents $\left(1-\chi^{E, {j-1}}\right)\Lc^{\Theta_{j-1}}\left(\Lc^{\Theta_{j-2}}\left(\cdots\Lc^{\Theta_1}\left(\chi_{s, t} \Lc^{\omega_\Xfh^{k-1}}(S)\right)\cdots\right)\right)$ and $\chi^{E, {j}}\Lc^{\Theta_{j+1}}\big(\cdots\chi^{E, {m-1}}\big(\Lc^{\Theta_m}$ $\left(f_*(\omega)^{l-1}\wedge dd^c \left(\chi^{E, k+2} u_f\right)\wedge \omega^{k-p-l+1}\right)\big)\cdots\big)$ have disjoint support and the distance between them is bounded below by $\left(\frac{1}{2}r_E\right)^{k+2}s$. Further, from Lemma \ref{lem:mass_regular_transform} and Lemma \ref{lem:tubular_closed_basic}, we see that the mass of the current $\left(1-\chi^{E, {j-1}}\right)\Lc^{\Theta_{j-1}}\left(\Lc^{\Theta_{j-2}}\left(\cdots\Lc^{\Theta_1}\left(\chi_{s, t} \Lc^{\omega_\Xfh^{k-1}}(S)\right)\cdots\right)\right)$ is bounded by
	$\|S\| t^{3/(2\delta+1)}$ up to a multiplicative constant independent of $n$, $i$, $s$, $t$, $\theta$ and $S$. Note that $u_f$ is H\"older continuous and $\chi^{E, k+2}$ is smooth on $X$. So, in the sense of currents, we may write
	\begin{align*}
		dd^c\left(\chi^{E, k+2} u_f\right)&=d \left( {\chi^{E, k+2}+u_f} \right)\wedge d^c\left(\chi^{E, k+2}+u_f\right) -d\chi^{E, k+2}\wedge d^c\chi^{E, k+2}-du_f\wedge d^cu_f\\
		&\quad\quad\quad\quad\quad\quad\quad\quad\quad\quad\quad\quad + u_fdd^c\chi^{E, k+2} + \chi^{E, k+2} dd^cu_f.
	\end{align*}
	The first three terms are positive or negative currents and the last two terms can be written as a difference of positive currents of mass under control. Using Stokes' theorem, we have
	\begin{align*}
		&\left|\int_Xd\left(\chi^{E, k+2}+u_f\right)\wedge d^c\left(\chi^{E, k+2}+u_f\right)\wedge \omega^{k-1}\right|=\left|\int_X\left(\chi^{E, k+2}+u_f\right)\wedge dd^c\left(\chi^{E, k+2}+u_f\right)\wedge \omega^{k-1}\right|\\
		&\quad\quad\quad\le c\left(\sup_X\left|u_f\right|+1\right)\left(\left\|\chi^{E, k+2}\right\|_{C^2}+\|f_*(\omega)\|+1\right),
	\end{align*}
	where $c>0$ is a constant independent of $n$, $i$, $s$, $t$, $\theta$, and $S$. The second and third terms $d\chi^{E, k+2}\wedge d^c\chi^{E, k+2}$ and $du_f\wedge d^cu_f$ can be treated in the same way. By Lemma \ref{lem:mass_regular_transform}, the current $\chi^{E,{j}}\Lc^{j+1}\big(\cdots\chi^{E,{m-1}}\big(\Lc^{\Theta_m} \big(f_*(\omega)^{l-1}\wedge dd^c \big(\chi^{E, k+2} u_f\big)\wedge \omega^{k-p-l+1}\big)\big)\cdots\big)$ can be written as a finite sum of positive or negative currents, the mass of each of which is bounded by $s^{-2}$ up to a multiplicative constant independent of $n$, $i$, $s$, $t$, $\theta$ and $S$. By Lemma \ref{lem:disjoint_estimate_1} and Lemma \ref{lem:disjoint_estimate_2}, 
	\begin{align*}
		&\bigg|\int \chi_{s, t} \Lc^{\omega_\Xfh^{k-1}}(S)\wedge \Lc^{\Theta_1}\big(\cdots\Lc^{\Theta_{j-2}}\big(\Lc^{\Theta_{j-1}}\big(\left(1-\chi^{E, {j-1}}\right)\Lc^{j}\big(\chi^{E, {j}}\Lc^{j+1}\big(\cdots\\
		&\quad\quad\quad\quad\quad\quad\chi^{E, {m-1}}\Lc^{\Theta_m}\left(f_*(\omega)^{l-1}\wedge dd^c \left(\chi^{E, k+2} u_f\right)\wedge \omega^{k-p-l+1}\right)\cdots\big)\big)\big)\big)\cdots\big)\bigg|\lesssim \|S\|s^{-2k}t^{3/(2\delta+1)}.
	\end{align*}
	The inequality $\lesssim$ means $\le$ up to a multiplicative constant independent of $n$, $i$, $s$, $t$, $\theta$, and $S$.
\end{proof}

\begin{proof}[Proof of Lemma \ref{lem:estimate_off_E-1}]
	Our proof is by induction. Observe that $\|S\|=\left|\left\langle S, \omega^{k-p}\right\rangle\right|=\left|\left\langle \alpha_S, \omega^{k-p}\right\rangle\right|\le c_m\left\|\alpha_S\right\|_\infty$ and therefore, Lemma \ref{lem:tubular_closed_basic} proves the case $l=0$. So, we assume that the statement is true for $l-1$ with $l\ge 1$. Let $\chi_{s, t}$ is the cut-off function in Lemma \ref{lem:cut-off}. Then, from the positivity of the integrand, we have
	{\small\begin{align}
			\notag\int_{V_t\setminus E_s} \Lc^{\omega_\Xfh^{k-1}}(S)\wedge \Lc^m\left(f_*(\omega)^l\wedge \omega^{k-p-l+1}\right)&\le \int \chi_{s, t} \Lc^{\omega_\Xfh^{k-1}}(S)\wedge \Lc^m\left(f_*(\omega)^l\wedge \omega^{k-p-l+1}\right)\\
			\label{integral:1}&= \int \chi_{s, t} \Lc^{\omega_\Xfh^{k-1}}(S)\wedge \Lc^m\left(f_*(\omega)^{l-1}\wedge \omega_f\wedge\omega^{k-p-l+1}\right)\\
			\label{integral:2}&\quad+ \int \chi_{s, t} \Lc^{\omega_\Xfh^{k-1}}(S)\wedge \Lc^m\left(f_*(\omega)^{l-1}\wedge dd^c u_f\wedge \omega^{k-p-l+1}\right),
	\end{align}}
	where the form $\omega_f$ and the function $u_f$ are as in Section \ref{sec:Holder}.
	Since $\omega$ is a K\"ahler form, the bound of the integral \eqref{integral:1} is obtained from the induction hypothesis.
	\begin{align*}
		&\left|\eqref{integral:1}\right|
		\lesssim 2\left\|\omega_f\right\|_\infty \left\|\alpha_S\right\|_\infty\left(r_Es\right)^{-2(k+l-1)-N_E}\left(t^{(\delta+1/2)^{-1}}\right)^{\beta_{l-1}}.
	\end{align*}
	
	We only need to estimate the integral \eqref{integral:2}. This is the part that we should be careful when compared to the works in \cite{DS09} and \cite{Ahn16}. We know that $u_f$ is $\left(c_f\left(r_E\left(\frac{1}{2}r_E\right)^{k+3}s\right)^{-N_E}, \delta^{-1}\right)$-H\"older continuous on $X\setminus E_{r_E\left(\frac{1}{2}r_E\right)^{k+3}s}$ from Lemma \ref{lem:Holder_conti_qpotential}. For sufficiently small $\gamma>0$ so that $\overline{\left(X\setminus E_{\left(\frac{1}{2}r_E\right)^{k+3}s}\right)_\gamma}\cap E_{r_E\left(\frac{1}{2}r_E\right)^{k+3}s}=\emptyset$ with $W_1=E_{\left(\frac{1}{2}r_E\right)^{k+3}s}$ and $W_2=E_{r_E\left(\frac{1}{2}r_E\right)^{k+3}s}$, Lemma \ref{lem:smooth_approx_Holder} implies that we can find a smooth function $\widetilde{\upsilon_\gamma}$ defined in a neighborhood $X\setminus E_{\left(\frac{1}{2}r_E\right)^{k+3}s}$ such that
	\begin{align*}
		\left\|\widetilde{\upsilon_\gamma}\right\|_{C^2, X\setminus E_{\left(\frac{1}{2}r_E\right)^{k+3}s}}\le c_H\gamma^{-2(k+1)}\textrm{ and } \left\|u_f-\widetilde{\upsilon_\gamma}\right\|_{\infty, X\setminus E_{\left(\frac{1}{2}r_E\right)^{k+3}s}}\le c_Hc_f\left(r_E\left(\frac{1}{2}r_E\right)^{k+3}s\right)^{-N_E}\gamma^{\delta^{-1}}
	\end{align*}
	where $c_H>0$ is the constant as in Lemma \ref{lem:smooth_approx_Holder} and is independent of $s$. 
	We define
	\begin{align*}
		\upsilon_\gamma=\chi^{E, k+2} u_f + \left(1-\chi^{E, k+2}\right) \widetilde{\upsilon_\gamma}.
	\end{align*}
	
	Then, \eqref{integral:2} is equal to the sum
	\begin{align}
		\label{integral:2-1}&\int \chi_{s, t} \Lc^{\omega_\Xfh^{k-1}}(S)\wedge \Lc^m\left(f_*(\omega)^{l-1}\wedge dd^c \upsilon_\gamma\wedge \omega^{k-p-l+1}\right)\\
		\label{integral:2-2}&\quad\quad\quad+\int \chi_{s, t} \Lc^{\omega_\Xfh^{k-1}}(S)\wedge \Lc^m\left(f_*(\omega)^{l-1}\wedge dd^c \left(u_f-\upsilon_\gamma\right)\wedge\omega^{k-p-l+1}\right).
	\end{align}
	The integral \eqref{integral:2-1} is equal to the sum
	\begin{align}
		\label{int:0-1}&\int \chi_{s, t} \Lc^{\omega_\Xfh^{k-1}}(S)\wedge \Lc^m\left(f_*(\omega)^{l-1}\wedge dd^c \left(\chi^{E, k+2} u_f\right)\wedge \omega^{k-p-l+1}\right)\\
		\label{int:0-2}&\quad+\int \chi_{s, t} \Lc^{\omega_\Xfh^{k-1}}(S)\wedge \Lc^m\left(f_*(\omega)^{l-1}\wedge dd^c\left(\left(1-\chi^{E, k+2}\right)\widetilde \upsilon_\gamma\right)\wedge \omega^{k-p-l+1}\right)
	\end{align}
	
	We can write the integral \eqref{int:0-1} as below:
	\begin{align*}
		\eqref{int:0-1}&=\sum_{j=1}^{m}\int \chi_{s, t} \Lc^{\omega_\Xfh^{k-1}}(S)\wedge \Lc^{\Theta_1}\big(\cdots\Lc^{\Theta_{j-2}}\big(\Lc^{\Theta_{j-1}}\big(\left(1-\chi^{E, {j-1}}\right)\Lc^{\Theta_j}\big(\chi^{E, {j}}\Lc^{\Theta_{j+1}}\big(\cdots\\
		&\quad\quad\quad\quad\quad\quad\quad\quad\quad\quad\quad\quad\chi^{E,{m-1}}\Lc^{\Theta_m}\left(f_*(\omega)^{l-1}\wedge dd^c \left(\chi^{E, k+2} u_f\right)\wedge \omega^{k-p-l+1}\right)\cdots\big)\big)\big)\big)\cdots\big).	
	\end{align*}
	Notice that $\chi^{E, m}dd^c\left(\chi^{E, k+2} u_f\right)=dd^c\left(\chi^{E, k+2} u_f\right)$ is used. Then, Lemma \ref{lem:integral_semi-regular_tr} implies that
	\begin{align*}
		|\eqref{int:0-1}|\lesssim \|S\|s^{-2k}t^{3/(2\delta+1)}.
	\end{align*}
	Since $\left\|\widetilde{\upsilon_\gamma}\right\|_{C^2, X\setminus E_{\left(\frac{1}{2}r_E\right)^{k+3}s}}\le c_H\gamma^{-2(k+1)}$ and $\left\| \chi^{E, k+2}\right\|_{C^2}\lesssim s^{-2}$, the induction hypothesis implies
	\begin{align*}
		|\eqref{int:0-2}|&
		\lesssim \gamma^{-2(k+1)}s^{-2}\cdot\left\|\alpha_S\right\|_\infty\left(r_Es\right)^{-2(k+l-1)-N_E}\left(t^{(\delta+1/2)^{-1}}\right)^{\beta_{l-1}}
	\end{align*}
	and therefore, we get
	\begin{align*}
		|\eqref{integral:2-1}|\lesssim \left\|\alpha_S\right\|_\infty\gamma^{-2(k+1)}s^{-2(k+l)-N_E}\left(t^{(\delta+1/2)^{-1}}\right)^{\beta_{l-1}}.
	\end{align*}
	
	Since $\Lc^m$ is closed, we have
	\begin{align*}
		\eqref{integral:2-2}=\int dd^c\chi_{s, t}\wedge \Lc^{\omega_\Xfh^{k-1}}(S)\wedge \Lc^m\left(\left(u_f-\upsilon_\gamma\right)f_*(\omega)^{l-1}\wedge \omega^{k-p-l+1}\right).
	\end{align*}
	Observe that
	\begin{align*}
		u_f-\upsilon_\gamma=\left(1- \chi^{E, k+2}\right)\left(u_f-\widetilde{\upsilon_\gamma}\right)\textrm{ and }\left\|u_f-\widetilde{\upsilon_\gamma}\right\|_{\infty, X\setminus E_{\left(\frac{1}{2}r_E\right)^{k+3}s}}\lesssim s^{-N_E}\gamma^{\delta^{-1}}.
	\end{align*}
	Since $\left\|\chi_{s, t}\right\|_\dsh\lesssim s^{-2}$, we can find two positive closed $(1, 1)$-currents $\vartheta_{s, t}^\pm$ such that $\vartheta_{s, t}^+-\vartheta_{s, t}^-=dd^c\chi_{s, t}$ and $\left\|\vartheta_{s, t}^\pm\right\|\lesssim s^{-2}$.
	Hence, since $\vartheta_{s, t}^\pm$, ${\omega_\Xfh^{k-1}}$, $S$, $f_*(\omega)$, $\omega$ are all positive and closed, cohomological arguments give us
	\begin{align*}
		|\eqref{integral:2-2}|&\le 2\left\|u_f-\upsilon_\gamma\right\|_\infty\left(\int \vartheta_{s, t}^+\wedge \Lc^{\omega_\Xfh^{k-1}}(S)\wedge \Lc^m\left(f_*(\omega)^{l-1}\wedge \omega^{k-p-l+1}\right)\right)\\
		&=2\left\|u_f-\upsilon_\gamma\right\|_\infty\left(\int \vartheta_{s, t}^+\wedge \Lc^{\omega_\Xfh^{k-1}}\left(\alpha_S\right)\wedge \Lc^m\left(\omega_f^{l-1}\wedge \omega^{k-p-l+1}\right)\right)\\
		&=2\left\|u_f-\upsilon_\gamma\right\|_\infty\left\|\alpha_S\right\|_\infty\left\|\omega_f\right\|^{l-1}\left(\int \vartheta_{s, t}^+\wedge \Lc^{\omega_\Xfh^{k-1}}\left(\omega^p\right)\wedge \Lc^m\left(\omega^{k-p}\right)\right)\\
		&\lesssim 2s^{-N_E}\gamma^{\delta^{-1}} \left\|\alpha_S\right\|_\infty\left\|\omega_f\right\|^{l-1}s^{-2}\lesssim  \left\|\alpha_S\right\|_\infty\gamma^{\delta^{-1}} s^{-2-N_E}.
	\end{align*}
	\medskip
	
	We take $\gamma=t^{(2(k+1)(2\delta+1))^{-l}\delta^{-(l-1)}}$. Then, from the above, we have
	\begin{align*}
		&\left|\eqref{integral:2-1}\right|\lesssim \left\|\alpha_S\right\|_\infty\gamma^{-2(k+1)}s^{-2(k+l)-N_E}\left(t^{(\delta+1/2)^{-1}}\right)^{\beta_{l-1}}\\
		&\quad\lesssim \left\|\alpha_S\right\|_\infty s^{-2(k+l)-N_E}t^{\frac{-2(k+1)}{(2(k+1)(2\delta+1))^{l}\delta^{l-1}}+\frac{1}{(2(k+1)(2\delta+1))^{l-1}\delta^{l-1}(\delta+1/2)}}
		\le \left\|\alpha_S\right\|_\infty s^{-2(k+l)-N_E}t^{\beta_l}
	\end{align*}
	and
	\begin{align*}
		\left|\eqref{integral:2-2}\right|&\lesssim \left\|\alpha_S\right\|_\infty s^{-N_E}\gamma^{\delta^{-1}} s^{-2}\lesssim \left\|\alpha_S\right\|_\infty s^{-2-N_E}t^{\frac{1}{(2(k+1)(2\delta+1))^l\delta^l}}\le \left\|\alpha_S\right\|_\infty s^{-2-N_E}t^{\beta_l}.
	\end{align*}
	Combining these two gives the estimate of \eqref{integral:2}, which proves the desired statement for $l$.
\end{proof}

\begin{remark}
	In the above proof, from the relationship between $s_{n,i}$ and $t_{n, i}$, our choice of $\gamma$ satisfies the desired smallness for the argument when $0<\varepsilon\ll 1$. So, the proof works.
\end{remark}

\subsubsection{Estimates of $\langle \chi_{n,i,1}\Lc^{u{\omega_\Xfh^{k-1}}}(S), \Lc^m(f_*(\omega^{k-p+1}))\rangle$} We prove the following lemma:
\begin{lemma}\label{lem:near_E_semiregular}
	Let $n$ and $i$ be integers such that $1\le i\le n$. Then, for $0<\varepsilon\ll 1$, we have
	\begin{align*}
		\left|\langle \chi_{n,i,1}\Lc^{u{\omega_\Xfh^{k-1}}}(S), \Lc^m(f_*(\omega^{k-p+1}))\rangle \right| \lesssim \left\|\alpha_S\right\|_\infty s_{n,i}^{-4k-N_E}t_{n,i}^{\beta_{k+1}/2}
	\end{align*}
	where $\beta_{k+1}=(2(k+1)(2\delta+1)\delta)^{-(k+1)}$ and the inequality $\lesssim$ means $\le$ up to a multiplicative constant independent of $n$, $i$, $\theta$, $S$ and $\alpha_S$.
\end{lemma}

Recall the definition of $u_\tau$ for $\tau\in\C^*$ with $|\tau|\ll 1$ in the beginning of Section \ref{sec:regularization}. We have
\begin{align*}
	(\log |\tau|-1)+\left(u-u_\tau\right)\le u\le 0
\end{align*}
and therefore,
{\small \begin{align}
		\notag&(\log|\tau|-1)\Lc^{{\omega_\Xfh^{k-1}}}(S)+\Lc^{(u-u_\tau){\omega_\Xfh^{k-1}}}(S)\le \Lc^{u{\omega_\Xfh^{k-1}}}(S)\le 0\quad\textrm{ and }\\
		\notag&(\log|\tau|-1)\left\langle \chi_{s_{n,i}, t_{n,i}}\Lc^{\omega_\Xfh^{k-1}}(S), \Lc^m\left(f_*\left(\omega^{k-p+1}\right)\right)\right\rangle+\left\langle \chi_{s_{n,i}, t_{n,i}}\Lc^{(u-u_\tau){\omega_\Xfh^{k-1}}}(S), \Lc^m\left(f_*\left(\omega^{k-p+1}\right)\right)\right\rangle\\
		\label{eq:approximation_uPhi}&\quad\quad\quad\quad\quad\quad\quad\quad\quad\quad\quad\quad\quad\quad\quad\quad\quad\quad\quad\quad\le \left\langle \chi_{s_{n,i}, t_{n,i}}\Lc^{u{\omega_\Xfh^{k-1}}}(S), \Lc^m\left(f_*\left(\omega^{k-p+1}\right)\right)\right\rangle\le 0
	\end{align}
}
The first term containing $(\log|\tau|-1)\Lc^{{\omega_\Xfh^{k-1}}}(S)$ can be deduced from the previous estimate and $u-u_\tau$ is negative and has support in a small neighborhood of $\widehat{\Delta}$. The estimates about the second term is essentially the same as \cite[Lemma 2.3.9, Lemma 2.3.10]{DS09}.
\begin{lemma} We have
	\begin{align*}
		\left|\left\langle\Lc^{(u-u_\tau){\omega_\Xfh^{k-1}}}(S), \Lc^m\left(\omega^{k-p+1}\right)\right\rangle\right|\lesssim \|S\||\tau|^{1/2}.
	\end{align*}
	The inequality $\lesssim$ means $\le$ up to a multiplicative constant independent of $\theta$, $S$ and $\tau$.
\end{lemma}

\begin{proof}
	From the negativity of $\Lc^{(u-u_\tau){\omega_\Xfh^{k-1}}}(S)$, we have
	\begin{align*}
		\left|\left\langle\Lc^{(u-u_\tau){\omega_\Xfh^{k-1}}}(S), \Lc^m\left(\omega^{k-p+1}\right)\right\rangle\right|\lesssim \left\|\Lc^m\left(\omega^{k-p+1}\right)\right\|_\infty\left|\left\langle\Lc^{(u-u_\tau){\omega_\Xfh^{k-1}}}(S), \omega^{k-p+1}\right\rangle\right|
	\end{align*}
	Since $\omega^{k-p+1}$ is smooth, so we have
	\begin{align*}
		\left\langle \Lc^{(u-u_\tau){\omega_\Xfh^{k-1}}}(S), \omega^{k-p+1}\right\rangle
		=\int_{X\times X\setminus \Delta} \pi_*\left[(u-u_\tau)\omega_\Xfh^{k-1}\right]\wedge\pi_1^*(S)\wedge\pi_2^*\left(\omega^{k-p+1}\right).
	\end{align*}
	Since we will not use the closedness, by disintegration as in Lemma \ref{lem:tubular_closed_basic}, we can write it as
	\begin{align*}
		\int_{x\in X} \left[\int_{y\in X\setminus \{x\}}\Psi_S(x)\righthalfcup\left[(u-u_\tau)\omega_\Xfh^{k-1}\right](x, y)\wedge\omega^{k-p+1}(y)\right]d\mu(x)
	\end{align*}
	in terms of Federer's notation. Note that the form $\Psi_S(x)\righthalfcup\left[\pi_*\left[(u-u_\tau){\omega_\Xfh^{k-1}}\right]\right]$ has a singularity of $-\left\|{\omega_\Xfh^{k-1}}\right\|_\infty\dist(x, y)^{2-2k}\log\dist(x, y)$ and its support lies inside a ball with center at $x$ and of radius $c\tau$ where $c>0$ is a uniform constant independent of $\tau$ and one might think of $\mu$ as the trace measure of $S$. Hence, $\int_X d\mu$ can be bounded in terms of $\|S\|$. Since $\|\Psi_S\|\le 1$, the integrand $\int_{y\in X\setminus \{x\}}\Psi_S(x)\righthalfcup\left[\pi_*\left[(u-u_\tau){\omega_\Xfh^{k-1}}\right]\right]\wedge\pi_2^*\left(\omega^{k-p+1}\right)$ is bounded by a uniform constant multiple of $\tau^{1/2}\left\|{\omega_\Xfh^{k-1}}\right\|_\infty$. So, it is proved. 
\end{proof}

\begin{lemma}\label{lem:u-u_t}
	For $l\in\{0, 1, \ldots, k-p+1\}$, we have
	\begin{align*}
		\left|\int \Lc^{(u-u_\tau){\omega_\Xfh^{k-1}}}(S)\wedge \Lc^m\left(f_*(\omega)^l\wedge \omega^{k-p-l+1}\right)\right|\lesssim \left\|\alpha_S\right\|_\infty|\tau|^{\frac{1}{2(4kd_{k})^l}}
	\end{align*}
	The inequality $\lesssim$ means $\le$ up to a multiplicative constant independent of $\theta$, $S$, $\alpha_S$ and $\tau$.
\end{lemma}

\begin{proof}
	We prove it by induction. We have $\|S\|=\left|\left\langle S, \omega^{k-p}\right\rangle\right|=\left|\left\langle \alpha_S, \omega^{k-p}\right\rangle\right|\le c_m\left\|\alpha_S\right\|_\infty$ The previous lemma corresponds to the case of $l=0$. Let $l\ge 1$. Assume that it is true for $l-1$.
	\begin{align}
		\notag&\int \Lc^{(u-u_\tau){\omega_\Xfh^{k-1}}}(S)\wedge \Lc^m\left(f_*(\omega)^l\wedge \omega^{k-p-l+1}\right)\\
		\label{integral:2-1-1}&=\int \Lc^{(u-u_\tau){\omega_\Xfh^{k-1}}}(S)\wedge \Lc^m\left(f_*(\omega)^{l-1}\wedge \omega_f\wedge \omega^{k-p-l+1}\right)\\
		\label{integral:2-2-1}&\quad\quad\quad +\int \Lc^{(u-u_\tau){\omega_\Xfh^{k-1}}}(S)\wedge \Lc^m\left(f_*(\omega)^{l-1}\wedge dd^cu_f\wedge \omega^{k-p-l+1}\right).
	\end{align}
	
	The estimate of \eqref{integral:2-1-1} comes from the induction hypothesis. 
	Since $f:X\to X$ can have multiplicity at most $d^k$, Lemma \ref{lem:Holder_conti_qpotential} implies that $u_f$ is $(K, d_k^{-1})$-H\"older continuous on $X$ for some $K>0$. We can use Lemma \ref{lem:smooth_approx_Holder} to find a smooth function $\upsilon_\gamma$ such that $\left\|u_f-\upsilon_\gamma\right\|_{\infty}\le K\gamma^{1/d_k}$ and $\left\|\upsilon_\gamma\right\|_{C^\infty}\lesssim K|\gamma|^{-2(k+1)}$ where $\gamma=|\tau|^{\frac{1}{2(4k)^ld_k^{l-1}}}$. Then, since $\Lc^m$ is closed, 
	\begin{align}
		\eqref{integral:2-2-1}&=\notag\int_{\widehat{X\times X}}(u-u_\tau){\omega_\Xfh^{k-1}}\wedge\Pi_1^*(S)\wedge \Pi_2^*\left(\Lc^m\left(f_*(\omega)^{l-1}\wedge dd^cu_f\wedge \omega^{k-p-l+1}\right)\right)\\
		&\label{integral:3-1}=\int_{\widehat{X\times X}}(u-u_\tau){\omega_\Xfh^{k-1}}\wedge\Pi_1^*(S)\wedge dd^c\Pi_2^*\left(\Lc^m\left(\left(u_f-\upsilon_\gamma\right)f_*(\omega)^{l-1}\wedge\omega^{k-p-l+1}\right)\right)\\
		&\label{integral:3-2}\quad\quad\quad+\int_{\widehat{X\times X}}(u-u_\tau){\omega_\Xfh^{k-1}}\wedge\Pi_1^*(S)\wedge \Pi_2^*\left(\Lc^m\left(f_*(\omega)^{l-1}\wedge dd^c(\upsilon_\gamma)\wedge \omega^{k-p-l+1}\right)\right)
	\end{align}
	Since the $*$-norm of $dd^c(u-u_\tau)$ is uniformly bounded independently of $\tau$ and $m_{\widehat{\Delta}}\omega_{\widehat{\mathfrak{X}}}+dd^cu, m_{\widehat{\Delta}}\omega_{\widehat{\mathfrak{X}}}+dd^cu_\tau\ge 0$, we use cohomological arguments to get an estimate of \eqref{integral:3-1} as follows:
	\begin{align*}
		&|\eqref{integral:3-1}|=\left|\int_{\widehat{X\times X}}dd^c(u-u_\tau)\wedge{\omega_\Xfh^{k-1}}\wedge\Pi_1^*(S)\wedge \Pi_2^*\left(\Lc^m\left(\left(u_f-\upsilon_\gamma\right)f_*(\omega)^{l-1}\wedge \omega^{k-p-l+1}\right)\right)\right|\\
		&\le2\|u_f-\upsilon_\gamma\|_\infty\left|\int_{\widehat{X\times X}}m_{\widehat \Delta}\omega_{\Xfh}\wedge{\omega_\Xfh^{k-1}}\wedge\Pi_1^*(S)\wedge \Pi_2^*\left(\Lc^m\left(f_*(\omega)^{l-1}\wedge \omega^{k-p-l+1}\right)\right)\right|\\
		&=2\|u_f-\upsilon_\gamma\|_\infty\left|\int_{\widehat{X\times X}}m_{\widehat \Delta}\omega_{\Xfh}\wedge{\omega_\Xfh^{k-1}}\wedge\Pi_1^*\left(\alpha_S\right)\wedge \Pi_2^*\left(\Lc^m\left(\omega_f^{l-1}\wedge \omega^{k-p-l+1}\right)\right)\right|\lesssim\left\|\alpha_S\right\|_\infty|\tau|^{\frac{1}{2(4kd_k)^l}}.
	\end{align*}
	
	The currents $S$, $f_*(\omega)$ and $\omega$ are all positive and $u-u_\tau\le 0$. So, the induction hypothesis gives
	\begin{align*}
		|\eqref{integral:3-2}|&
		\lesssim \|\upsilon_\gamma\|_{C^2}\left|\int_{\widehat{X\times X}}(u-u_\tau){\omega_\Xfh^{k-1}}\wedge\Pi_1^*(S)\wedge \Pi_2^*\left(\Lc^m\left(f_*(\omega)^{l-1}\wedge \omega^{k-p-l+2}\right)\right)\right|\\
		&\lesssim |\gamma|^{-2(k+1)} \left\|\alpha_S\right\|_\infty|\tau|^{\frac{1}{2(4kd_{k})^{l-1}}}=\left\|\alpha_S\right\|_\infty|\tau|^{\frac{2(k-1)}{2(4k)^ld_k^{(l-1)}}}<\left\|\alpha_S\right\|_\infty|\tau|^{\frac{1}{2(4kd_k)^l}}.
	\end{align*}
\end{proof}

\begin{proof}
	[Proof of Lemma \ref{lem:near_E_semiregular}]
	The current $\Lc^{u{\omega_\Xfh^{k-1}}}(S)$ is negative and the current $\Lc^m(f_*(\omega^{k-p+1}))$ is positive. We take $l=k-p+1$ and $\tau=e^{1-t_{n,i}^{-\beta_{k-p+2}/2}}$. Then, for $0<\varepsilon\ll 1$, \eqref{eq:approximation_uPhi}, Lemma \ref{lem:estimate_off_E-1} and Lemma \ref{lem:u-u_t} give us 
	\begin{align*}
		0&\ge\left\langle \chi_{n,i,1}\Lc^{u{\omega_\Xfh^{k-1}}}(S), \Lc^m\left(f_*\left(\omega^{k-p+1}\right)\right)\right\rangle\ge\left\langle \chi_{s_{n,i}, t_{n,i}}\Lc^{u{\omega_\Xfh^{k-1}}}(S), \Lc^m\left(f_*(\omega)^{k-p+1}\right)\right\rangle\\
		&\ge -\left(t_{n,i}^{-\beta_{k-p+2}/2}\cdot s^{-4k-N_E}\left(t_{n,i}^{(\delta+1/2)^{-1}}\right)^{\beta_{k-p+1}}+e^{\frac{1-t_{n,i}^{-\beta_{k-p+2}/2}}{2(4kd_k)^{k-p+1}}}\right)\left\|\alpha_S\right\|_\infty\\
		&\gtrsim-\left(s^{-4k-N_E}t_{n,i}^{\beta_{k-p+2}}+\frac{2(4kd_k)^{k-p+1}}{{t_{n,i}^{-\beta_{k-p+2}/2}-1}}\right)\left\|\alpha_S\right\|_\infty\gtrsim-s_{n,i}^{-4k-N_E}t_{n,i}^{\beta_{k+1}/2}\left\|\alpha_S\right\|_\infty.
	\end{align*}
\end{proof}
\subsubsection{Estimates of $d^{-i}\langle \chi_{n,i,1}U_{L^{n-i}(S-\alpha_S)},\Lambda(R_{n,i-1})-R_{n,i}\rangle$}
\begin{lemma}\label{Wni1}For all $n$ and $i$ be integers such that $1\le i\le n$ and for $0<\varepsilon\ll 1$, we have
	\begin{align*}
		&\left| \left\langle \chi_{n,i,1}U_{L^{n-i}(S-\alpha_S)},\Lambda(R_{n,i-1})-R_{n,i}\right\rangle\right| \lesssim \varepsilon.
	\end{align*}
	Here, the inequality $\lesssim$ means $\le$ up to a multiplicative constant independent of $n$, $i$, $S$ and $\alpha_S$.
\end{lemma} 

\begin{proof}
	Lemma \ref{lem:mass_estimates} says that $R_{n,i}=(\Lambda(R_{n,i-1}^+))_{\varepsilon_{n, i}}^++(\Lambda(R_{n,i-1}^-))_{\varepsilon_{n, i}}^--(\Lambda(R_{n,i-1}^+))_{\varepsilon_{n, i}}^--(\Lambda(R_{n,i-1}^-))_{\varepsilon_{n, i}}^+$ for $i\ge 2$. When $i= 1$, we simply remove terms containing $R_{n,0}^-$. Due to Proposition \ref{prop:Green_kernel}, the Green potential kernel can be written as $\Lc^K=\Lc^{K_+}-\Lc^{K_-}$. Therefore, $|\langle \chi_{n,i,1}U_{L^{n-i}(S-\alpha_S)},\Lambda(R_{n,i-1})-R_{n,i}\rangle|$ is bounded
	by the sum of $\left|\langle \chi_{n,i,1}\Lc^{K_\pm}(L^{n-i}(S)),\Lambda(R_{n,i-1}^\pm)\rangle\right|$ and $\big|\langle \chi_{n,i,1}\Lc^{K_\pm}(L^{n-i}(S)),$ $(\Lambda(R_{n,i-1}^\pm))_{\varepsilon_{n, i}}^\pm\rangle\big|$.
	The forms $K_\pm$ are negative and the current $L^{n-i}(S)$ is positive.
	When $i\ge 2$, each term can be estimated by Lemma \ref{lem:estimate_off_E-1}, Lemma \ref{lem:near_E_semiregular} and Lemma \ref{lem:mass_estimates} as below:
	\begin{align*}
		&\left|\left\langle \chi_{n,i,1}\Lc^{K_+}\left(L^{n-i}(S)\right),\Lambda\left(R_{n,i-1}^+\right)\right\rangle\right|=\left|\left\langle \chi_{n,i,1}\Lc^{u\left(\eta+m_K\omega_{\Xfh}^{k-1}\right)-m_K\omega_{\Xfh}^{k-1}}\left(L^{n-i}(S)\right),\Lambda\left(R_{n,i-1}^+\right)\right\rangle\right|\\
		&\quad\lesssim \left\|R^+_{n,i-1}\right\|_\infty(\|\eta\|_\infty + m_K)\left|\left\langle \chi_{n,i,1}\left(\Lc^{u\omega_{\Xfh}^{k-1}}\left(L^{n-i}(S)\right)+\Lc^{\omega_{\Xfh}^{k-1}}\left(L^{n-i}(S)\right)\right),\Lambda\left(\omega^{k-p+1}\right)\right\rangle\right|\\
		&\quad\lesssim \left\|R^+_{n,i-1}\right\|_\infty\left|\left\langle \chi_{n,i,1}\left(\Lc^{u\omega_{\Xfh}^{k-1}}\left(L^{n-i}(S)\right)+\Lc^{\omega_{\Xfh}^{k-1}}\left(L^{n-i}(S)\right)\right),f_*(\omega)^{k-p+1}\right\rangle\right|\\
		&\quad\lesssim \left\|R_{n,i-1}^+\right\|_\infty\|L^{n-i}(\alpha_S)\|_\infty s_{n,i}^{-4k-N_E}t_{n,i}^{\beta_{k+1}/2}\\
		&\quad\lesssim \|R\|_{C^1}(C_m)^{i-1}\left(\prod_{j=1}^{i-1}\varepsilon_{n,j}\right)^{-6k^2}(1+\|f\|_{C^1})^{kn}\|\alpha_S\|_\infty s_{n,i}^{-4k-N_E}t_{n,i}^{\beta_{k+1}/2}\\
		&\quad\lesssim \left((C_m+1)(1+\|f\|_{C^1})^k\right)^n \varepsilon_{n, i-1}^{-12k^2}s_{n,i}^{-4k-N_E}t_{n,i}^{\beta_{k+1}/2}\\
		&\quad\lesssim \left((C_m+1)(1+\|f\|_{C^1})^k\right)^n\varepsilon^{nC(4k+N_E)\left[(40k^2\delta)^{6k(i-1)}\left(\frac{(40k^2\delta)^{6k}}{20k(2(k+1)(2\delta+1)\delta)^{k+1}}-12k^2\right)-i\right]}\\
		&\quad\le \left((C_m+1)(1+\|f\|_{C^1})^k\right)^n\varepsilon^{nC(4k+N_E)\left[(40k^2\delta)^{6k(i-1)}\left((40k^2\delta)^{3k}-12k^2\right)-i\right]}<\varepsilon
	\end{align*}
	for all sufficiently small $\varepsilon>0$. For $i=1$, it can be done in the same way. The regularizing operation $(\cdot)_{\varepsilon_{n, i}}^+$ in Section \ref{sec:regularization} is the $(k+2)$-times compositions of either $\Lc_\theta^+$ or $\Lc^-$. When $\varepsilon>0$ is small enough, the real numbers $\varepsilon_{n, i}$ and $t_{n,i}$ are sufficiently small compared to $s_{n,i}$. Thus, we get
	\begin{align*}
		&\left|\left\langle \chi_{n,i,1}\Lc^{K_+}\left(L^{n-i}(S)\right),\left(\Lambda\left(R_{n,i-1}^+\right)\right)_{\varepsilon_{n, i}}^+\right\rangle\right|\lesssim \left\|R^+_{n,i-1}\right\|_\infty\left|\left\langle \chi_{n,i,1}\Lc^{K_+}\left(L^{n-i}(S)\right),\left(\Lambda\left(\omega^{k-p+1}\right)\right)_{\varepsilon_{n, i}}^+\right\rangle\right|\\
		&\quad\lesssim \left\|R^+_{n,i-1}\right\|_\infty\left|\left\langle \chi_{n,i,1}\Lc^{K_+}\left(L^{n-i}(S)\right),\left(f_*(\omega)^{k-p+1}\right)_{\varepsilon_{n, i}}^+\right\rangle\right|<\varepsilon.
	\end{align*}
	in the same way as previously. All other terms can be computed in the same way. For the terms containing $\alpha_S$, we can apply the same argument as well. Hence, we have
	\begin{align*}
		&\left|\left\langle \chi_{n,i,1}\Lc^{K_\pm}\left(L^{n-i}(S)\right),\Lambda\left(R_{n,i-1}^\pm\right)\right\rangle\right|\textrm{ , }\left|\left\langle \chi_{n,i,1}\Lc^{K_\pm}\left(L^{n-i}(S)\right),\left(\Lambda\left(R_{n,i-1}^\pm\right)\right)_{\varepsilon_{n, i}}^\pm\right\rangle\right|\lesssim \varepsilon 
	\end{align*}
	for all integers $n$ and $i$ such that $1\le i\le n$ and for all sufficiently small $\varepsilon>0$. Here, the inequality $\lesssim$ means $\le$ up to a multiplicative constant independent of $n$, $i$ and $S$.
\end{proof}

\subsection{Estimates of $d^{-i}\langle \chi_{n,i,2}U_{L^{n-i}(S-\alpha_S)},\Lambda(R_{n,i-1})-R_{n,i}\rangle$ in $W_{n,i,2}$} We will use the property that in $W_{n,i,2}$, $L^{n-i}(S)$ is smooth.
\begin{lemma}\label{lem:C^1_Wni2}
	Assume that $S\in\Dc_p$ is smooth in $E_{s/{r_E^3}}$ for $s>0$. The currents $\chi^E_s\Lc^{u{\omega_\Xfh^{k-1}}}(S)$ and $\chi^E_s\Lc^{{\omega_\Xfh^{k-1}}}(S)$ is $C^1$. Furthermore, we have
	\begin{align*}
		&\|\chi^E_s\Lc^{u{\omega_\Xfh^{k-1}}}(S)\|_{C^1}\lesssim s^{1-2k}(\|S\|_{\infty, E_{s/{r_E^2}}}+\|S\|_*)\textrm{ and }\|\chi^E_s\Lc^{{\omega_\Xfh^{k-1}}}(S)\|_{C^1}\lesssim s^{1-2k}(\|S\|_{\infty, E_{s/{r_E^2}}}+\|S\|_*).
	\end{align*}
	Here, the inequality $\lesssim$ means $\le$ up to a multiplicative constant independent of $s$ and $S$.
\end{lemma}

\begin{proof}
	The same proof works for both. So, we only consider $\chi^E_s\Lc^{u{\omega_\Xfh^{k-1}}}(S)$. We have
	\begin{align*}
		\chi^E_s\Lc^{u{\omega_\Xfh^{k-1}}}(S)=\chi^E_s\Lc^{u{\omega_\Xfh^{k-1}}}\left(\chi^E_{s/r_E^2}S\right)+\chi^E_s\Lc^{u{\omega_\Xfh^{k-1}}}\left(\left(1-\chi^E_{s/r_E^2}\right)S\right)
	\end{align*}
	
	As in \cite[Section 2.3]{DS10-1}, $\pi_*(u{\omega_\Xfh^{k-1}})$ is smooth outside $\Delta$ and its gradient satisfies $\left|\nabla \pi_*(u{\omega_\Xfh^{k-1}})\right|\lesssim (\dist(\cdot, \Delta))^{1-2k}$. Hence, we have
	\begin{align*}
		\left\|\chi^E_s\Lc^{u{\omega_\Xfh^{k-1}}}\left(\chi^E_{s/r_E^2}S\right)\right\|_{C^1}&\le 2\left\|\chi^E_s\right\|_{C^1}\left\|\Lc^{u{\omega_\Xfh^{k-1}}}\left(\chi^E_{s/r_E^2}S\right)\right\|_{C^1}\lesssim \left\|\chi^E_s\right\|_{C^2}\left\|\chi^E_{s/r_E^2}S\right\|_\infty\le s^{-2}\|S\|_{\infty, E_{s/r_E^2}}
	\end{align*}
	
	Let $\varphi$ be a smooth test form on $X$. Then, we have
	\begin{align*}
		&\left\langle \chi^E_s\Lc^{u{\omega_\Xfh^{k-1}}}\left(\left(1-\chi^E_{s/r_E^2}\right)S\right), \varphi\right\rangle=\left\langle\Lc^{u{\omega_\Xfh^{k-1}}}\left(\left(1-\chi^E_{s/r_E^2}\right)S\right), \chi^E_s\varphi\right\rangle\\
		&=\int_X\chi^E_s\left[\int_{X\setminus \{z\}} \pi_*(u{\omega_\Xfh^{k-1}})(z, \xi)\wedge \left(\left(1-\chi^E_{s/r_E^2}\right)S\right)(\xi)\right]\wedge \varphi.
	\end{align*}
	Since $\dist\left(\supp \chi_s^E, \supp \left(1-\chi^E_{s/r_E^2}\right)\right)\ge \left(r_E^{-1}-1\right)s$, again according to \cite[Section 2.3]{DS10-1}, we see that the current $\displaystyle \chi^E_s\left[\int_{X\setminus \{z\}} \pi_*(u{\omega_\Xfh^{k-1}})(z, \xi)\wedge \left(\left(1-\chi^E_{s/r_E^2}\right)S\right)(\xi)\right]$ is smooth and
	\begin{align*}
		\left\|\chi^E_s\left[\int_{X\setminus \{z\}} \pi_*(u{\omega_\Xfh^{k-1}})(z, \xi)\wedge \left(\left(1-\chi^E_{s/r_E^2}\right)S\right)(\xi)\right]\right\|_{C^1}\lesssim
		s^{1-2k}\left\|S\right\|_*.
	\end{align*}
\end{proof}

\begin{lemma}\label{lem:Lem_Cor_Wni2} Let $\widetilde{R}\in\Dc_{k-p+1}$. Under the hypotheses in Lemma \ref{lem:C^1_Wni2}, for $0<\theta\ll 1$, we have
	\begin{align*}
		&\left|\langle \chi^E_s\Lc^{u{\omega_\Xfh^{k-1}}}(S), \widetilde{R}-(\widetilde{R})_\theta\rangle\right|\lesssim s^{1-2k}\theta\|S\|_{\infty, E_{s/{r_E^2}}}\|\widetilde{R}\|_*\quad\textrm{ and }\\
		&\quad\quad\quad\quad\quad\quad\quad\left|\langle \chi^E_s\Lc^{{\omega_\Xfh^{k-1}}}(S), \widetilde{R}-(\widetilde{R})_\theta\rangle\right|\lesssim s^{1-2k}\theta\|S\|_{\infty, E_{s/{r_E^2}}}\|\widetilde{R}\|_*.
	\end{align*}
	Here, the inequality $\lesssim$ means $\le$ up to a multiplicative constant independent of $s$, $S$ and $\theta$.
\end{lemma}

\begin{proof} As previously, we consider the first estimate. The same applies to the second. We have
	\begin{align*}
		\langle \chi^E_s\Lc^{u{\omega_\Xfh^{k-1}}}(S), \widetilde{R}-(\widetilde{R})_\theta\rangle&=\langle \chi^E_s\Lc^{u{\omega_\Xfh^{k-1}}}(S)-\left(\chi^E_s\Lc^{u{\omega_\Xfh^{k-1}}}(S)\right)_\theta, \widetilde{R}\rangle.
	\end{align*}
	Lemma \ref{lem:C^1_Wni2} implies that $\chi^E_s\Lc^{u{\omega_\Xfh^{k-1}}}(S)$
	is a form with $C^1$-coefficients and Lemma \ref{lem:reg_smooth_estimate} implies
	\begin{align*}
		&\left\|\chi^E_s\Lc^{u{\omega_\Xfh^{k-1}}}(S)-\left(\chi^E_s\Lc^{u{\omega_\Xfh^{k-1}}}(S)\right)_\theta\right\|_{\infty}\lesssim s^{1-2k}\theta(\|S\|_{\infty, E_{s/{r_E^2}}}+\|S\|_*).
	\end{align*}
	Hence, we get the estimate.
\end{proof}

\begin{corollary}\label{Cor:Wni2}
	Assume that $S$ be as in Theorem \ref{thm:main}. For $0<\varepsilon\ll 1$, we have
	\begin{align*}
		\left|\left\langle \chi_{n,i,2}U_{L^{n-i}(S-\alpha_S)},\Lambda(R_{n,i-1})-R_{n,i}\right\rangle\right|<(\|S\|_{\infty, E_{s_0}}+\|\alpha_S\|_\infty))\varepsilon.
	\end{align*}
	Here, the inequality $\lesssim$ means $\le$ up to a multiplicative constant independent of $n$, $i$ and $S$.
\end{corollary}

\begin{proof}
	Notice that $U_{L^{n-i}(S-\alpha_S)}$ is the difference of $\Lc^{K_+}\left(L^{n-i}(S-\alpha_S)\right)-\Lc^{K_-}\left(L^{n-i}(S-\alpha_S)\right)$ where $K_+= u\left(\eta+m_K\omega_{\Xfh}^{k-1}\right)-\left(m_K\omega_{\Xfh}^{k-1}\right)$ and $K_-= u\left(m_K\omega_{\Xfh}^{k-1}\right)-\left(\delta_\Xfh + m_K\omega_{\Xfh}^{k-1}\right)$. From Lemma \ref{lem:Lem_Cor_Wni2} and Lemma \ref{lem:mass_estimates}, we get
	\begin{align*}
		&\left|\left\langle \chi_{n,i,2}U_{L^{n-i}(S-\alpha_S)},\Lambda(R_{n,i-1})-R_{n,i}\right\rangle\right| =\left|\left\langle \chi_{n,i,2}U_{L^{n-i}(S-\alpha_S)},\Lambda(R_{n,i-1})-\left(\Lambda(R_{n,i-1})\right)_{\varepsilon_{n, i}}\right\rangle\right|\\
		&\quad\lesssim s_{n,i}^{1-2k}\varepsilon_{n, i}\left(\left\|L^{n-i}(S-\alpha_S)\right\|_{\infty, E_{s/{r_E^2}}}+\left\|L^{n-i}(S)\right\|+\left\|L^{n-i}(\alpha_S)\right\|\right)\|\Lambda(R_{n,i-1})\|_*.
	\end{align*}
	By Lemma \ref{lem:smooth_region}, for $0<\varepsilon\ll 1$, $S$ is smooth in $E_{\varepsilon/{r_E^3}}$ and $L^{n-i}(S-\alpha)$ is smooth in $E_{s_{n,i}/{r_E^3}}$. From direct computations, we get
	\begin{align*}
		\left\| L^{n-i}(S-\alpha_S)\right\|_{\infty, E_{s_{n,i}/{r_E^3}}}\le (1+\|f\|_{C^1})^{k(n-i)}\|S-\alpha_S\|_{\infty, E_{s_0}}
	\end{align*}
	where $S$ is smooth in $E_{s_0}$ for some fixed $s_0>0$. From \eqref{cond:mass_condition} and Lemma \ref{lem:mass_estimates}, we have
	\begin{align*}
		&\left|\langle \chi_{n,i,2}U_{L^{n-i}(S-\alpha_S)},\Lambda(R_{n,i-1})-R_{n,i}\rangle\right|\\
		&\quad\lesssim ((1+\|f\|_{C^1})^{kn}\|S-\alpha_S\|_{\infty, E_{s_0}}+2c_m\left(\frac{5}{4}\right)^{N_fn}\|\alpha_S\|_\infty)\left(\left\|\Lambda(R^+_{n,i-1})\right\|+\left\|\Lambda(R^-_{n,i-1})\right\|\right)s_{n,i}^{1-2k}\varepsilon_{n, i}\\
		&\quad \lesssim \left(\left(\left(\frac{5}{4}\right)^{N_f}+\|f\|_{C^1}\right)^{kn}\left(\|S\|_{\infty, E_{s_0}}+\|\alpha_S\|_\infty\right)\right)c_m\left(\frac{5}{4}\right)^{N_f}C_m^i \left(\prod_{j=1}^{i-1}\varepsilon_{n,j}\right)^{-6k^2}s_{n,i}^{1-2k}\varepsilon_{n, i}\\
		&\quad \lesssim \left(\|S\|_{\infty, E_{s_0}}+\|\alpha_S\|_\infty\right)\left(\left(\frac{5}{4}\right)^{N_f}+\|f\|_{C^1}\right)^{kn}(C_m+1)^n \varepsilon_{n,i-1}^{-12k^2}s_{n,i}^{1-2k}\varepsilon_{n, i}\lesssim (\|S\|_{\infty, E_{s_0}}+\|\alpha_S\|_\infty))\varepsilon 
	\end{align*}
	for $0<\varepsilon\ll 1$ where $R_{n,i-1}^\pm$ are smooth positive closed currents in Lemma \ref{lem:mass_estimates}.
	By plugging-in $s_{n,i}$ and $\varepsilon_{n, i}$, we see the desired estimate.
\end{proof}

\subsection{Estimates of $d^{-i}\langle \chi_{n,i,3}U_{L^{n-i}(S-\alpha_S)},\Lambda(R_{n,i-1})-R_{n,i}\rangle$ in $W_{n,i,3}$}
In this region, $\Lambda(R_{n,i})$ is smooth.

\begin{lemma}[Lemma 5.4.7 in \cite{DS09}]\label{lem:Calpha_estimate_Wi2}
	Let $\alpha\geq 0$. For $0<t\ll 1$, we have
	\begin{displaymath}
		\|\Lambda(R)\|_{\cali{C}^\alpha({V_{t}}^c)}\lesssim\|R\|_{\cali{C}^\alpha}t^{-(4+\alpha)k},
	\end{displaymath}
	for any smooth form $R$ of bidegree $(p, p)$ with $0\leq p\leq k$. The inequality is up to a multiplicative constant independent of $t$ and $R$.
\end{lemma}

\begin{corollary}\label{Cor:Wni3}
	For $0<\varepsilon\ll 1$, we have
	\begin{align*}
		d^{-i}\langle \chi_{n,i,3}U_{L^{n-i}(S-\alpha_S)},\Lambda(R_{n,i-1})-R_{n,i}\rangle\lesssim \varepsilon.
	\end{align*}
	Here, the inequality $\lesssim$ means $\le$ up to a multiplicative constant independent of $n$, $i$ and $S$.
\end{corollary}
\begin{proof}
	Note that the $*$-norm of the Green potential is uniformly bounded if the mass of $S$ is bounded. Indeed, for $S\in\Cc_p$, we have
	\begin{align*}
		\left|\left\langle\Lc^{K_\pm}(S), \omega^{k-p+1}\right\rangle\right|=\left|\left\langle S, \Lc^{K_\pm}\left(\omega^{k-p+1}\right)\right\rangle\right|\lesssim \max\left\{\left\|\Lc^{K_\pm}\left(\omega^{k-p+1}\right)\right\|_\infty\right\}\|S\|
	\end{align*}
	and $U_{S}=\Lc^{K_+}(S)-\Lc^{K_-(S)}$.
	So, from \eqref{cond:mass_condition}, Lemma \ref{lem:Calpha_estimate_Wi2} and Lemma \ref{lem:mass_estimates}, we get
	\begin{align*}
		&\left|d^{-i}\left\langle \chi_{n,i,3}U_{L^{n-i}(S-\alpha_S)},\Lambda(R_{n,i-1})-R_{n,i}\right\rangle\right|\lesssim \|\Lambda(R_{n,i-1})-R_{n,i}\|_{\infty, W_{n,i,3}}\|L^{n-i}(S-\alpha_S)\|_*\\
		&\quad\lesssim\|\Lambda(R_{n,i-1})\|_{C^1, W_{n,i,3}}\varepsilon_{n, i}\|L^{n-i}(S-\alpha_S)\|_*\lesssim\|R_{n,i-1}\|_{C^1}{t_{n,i}}^{-5k}\varepsilon_{n, i}\|L^{n-i}(S-\alpha_S)\|_*\\
		&\quad\lesssim \|R\|_{C^1}(C_m)^{i-1}\left(\prod_{j=1}^{i-1}\varepsilon_{n,j}\right)^{-6k^2} \varepsilon_{n, i}^{1/2}\left( 3c_m\left(\frac{5}{4}\right)^{N_fn}\|\alpha_S\|_\infty \right)\quad\lesssim C_m^n\left(\frac{5}{4}\right)^{N_fn}\varepsilon_{n, i-1}^{-12k^2}\varepsilon_{n, i}^{1/2}\le \varepsilon
	\end{align*}
	for $0<\varepsilon\ll 1$.
\end{proof}

From Lemma \ref{Wni1}, Corollary \ref{Cor:Wni2}, Corollary \ref{Cor:Wni3}, we have
\begin{lemma}
	For $0<\varepsilon\ll 1$, we have
	\begin{align*}
		d^{-i}\langle U_{L^{n-i}(S-\alpha_S)},\Lambda(R_{n,i-1})-R_{n,i}\rangle\lesssim \varepsilon.
	\end{align*}
	Here, the inequality $\lesssim$ means $\le$ up to a multiplicative constant independent of $n$, $i$ and $S$.
\end{lemma}

\section{Estimates of $d^{-n}\langle U_{S-\alpha_S},R_{n,n}\rangle$}\label{sec:(2)}
For this estimate, we use a version of exponential estimate \cite[Theorem 3.2.6]{DS10-1}.

\begin{theorem}
	[Theorem 3.2.6 in \cite{DS10-1}]\label{thm:exponential_estimate} Let $S\in\Dc_p$ and $\Uc_S$ be the $\beta$ normalized superpotential of $S$. Then we have for $R$ smooth in $\Dc_{k-p+1}^0$ with $\|R\|_*\le 1$,
	\begin{align*}
		|\Uc_S(R)|\le c\|S\|_*(1+\log^+\|R\|_{C^1}),
	\end{align*}
	where $\log^+:=\max\{\log, 0\}$ and $c>0$ is a constant independent of $S$ and $R$.
\end{theorem}

\begin{corollary} For $0<\varepsilon\ll 1$, we have
	\begin{align*}
		\left|d^{-n}\langle U_{S-\alpha_S},R_{n,n}\rangle\right|\lesssim \|S\|_*(1+\log^+\|R_{n,n}\|_{C^1})\lesssim nd^{n/4}(-\log\varepsilon)\|S\|_*.
	\end{align*}
	Here, the inequality $\lesssim$ means $\le$ up to a multiplicative constant independent of $n$ and $S$.
\end{corollary}

\begin{proof}
	Let $R_\pm$ be smooth positive closed currents such that $R\in\Dc_{k-p+1}^0$, $R=R_+-R_-$ and $\|R_\pm\|_{C^1}\le 1$. Then, from Proposition \ref{prop:C1_estimate_regularization} and Lemma \ref{lem:mass_estimates}, we get 
	\begin{align*}
		&\left|d^{-n}\langle U_{S-\alpha_S},R_{n,n}\rangle\right|\lesssim d^{-n}\|S\|_*\left(1+\log^+\|R_{n,n}\|_{C^1}\right)\le d^{-n}\|S\|_*\left(1+\log^+(\|(R_+)_{n,n}\|_{C^1} + \|(R_-)_{n,n}\|_{C^1})\right)\\
		&\lesssim d^{-n}\left|1+2\log\left( 2\|R\|_{C^1} C_m^n\left(\prod_{j=1}^n\varepsilon_{n,j}\right)^{-6k^2} \right)\right|\|S\|_*\le d^{-n}\left|1+2\log\left( 2\|R\|_{C^1} C_m^n\varepsilon_{n,n}^{-12k^2} \right)\right|\|S\|_*\\
		&\lesssim \frac{12k^2nC(4k+N_E)(40k^2\delta)^{6kn}}{d^n}(-\log \varepsilon)\|S\|_*\lesssim n\left(\frac{(40k^2\delta)^{6k}}{d}\right)^n(-\log \varepsilon)\|S\|_*\lesssim nd^{n/4}(-\log\varepsilon)\|S\|_*
	\end{align*}
	for $0<\varepsilon\ll 1$. The last inequality is from \eqref{cond:delta_condition}.
\end{proof}

\section{Green currents and (almost) simple action}\label{sec:Green}
On a general compact K\"ahler manifold, the convergence of the sequence $d_p^{-n}(f^n)^*\omega^p$ is not clear. For the existence and construction of the Green current, see \cite{DS05}, \cite{DS10-1}, \cite{dTD}, \cite{DNV}. In this section, we summarize some related results 
%
from \cite{DNV}. Let $X$ and $f$ be as in Theorem \ref{thm:main}. We say that the action of $f$ on cohomology is simple if there is an integer $0\le \pmain\le k$ such that
\begin{enumerate}
	\item the dynamical degree $d_{\pmain}$ of order $\pmain$ of $f$ is strictly larger than the other dynamical degrees,
	\item $d_{\pmain}$ is a simple eigenvalue of $f^*$ on $H^{\pmain, \pmain}(X, \R)$, and
	\item the other (real or complex) eigenvalues of this operator have modulus strictly smaller than $d_{\pmain}$.
\end{enumerate}

Under these conditions, we have the existence of the Green current $T^+$.
\begin{proposition}[Proposition 5.5 in {\cite{DNV}}]
	\label{prop:construction_Green}Let $p=\pmain$. The sequence $d_p^{-n}(f^n)^*(\omega^p)$ converges weakly to a positive closed $(p, p)$-current $T^+$ as $n$ tends to infinity. The sequence $d_p^{-n}(f^n)_*(\omega^{k-p})$ converges weakly to a positive closed $(k-p, k-p)$-current $T^-$ as $n$ tends to $\infty$. Moreover, $T^+$ and $T^-$ have continuous superpotentials.
\end{proposition}

\cite[Lemma 4.1]{AV} implies the following convergence.
\begin{corollary}
	Let $p=\pmain$. We have $d_p^{-n}(f^n)^*(\alpha_S)\to c_{\alpha_S} T^+$ exponentially fast where $c_{\alpha_S}=\frac{\langle \alpha_S, T^-\rangle}{\langle \omega^p, T^-\rangle}$.
\end{corollary}

Together with Theorem \ref{thm:main}, we get
\begin{theorem} Let $X$ and $f$ be as in Theorem \ref{thm:main}. Assume further that the action of $f$ is simple and $p=\pmain$. Then, there exists a proper analytic subset $E$ invariant under $f$ such that if $S\in\Cc_p$ has a smooth representation in a neighborhood of $E$, then
	\begin{align*}
		d_p^{-n}(f^n)^*S \to c_ST^+
	\end{align*}
	exponentially fast where $c_S=\frac{\langle S, T^-\rangle}{\langle \omega^p, T^-\rangle}$. In particular, if $H$ is an analytic subset of pure dimension $k-p$ such that $H\cap E=\emptyset$, then the sequence $d_p^{-n}(f^n)^*[H]$ converges to $c_HT^+$ where $c_H=\frac{\langle [H], T^-\rangle}{\langle \omega^p, T^-\rangle}$.
\end{theorem}

\begin{remark}
	If we do not require the existence of the limit $T^-:=d_p^{-n}(f^n)_*(\omega^{k-p})$, the convergence of $d_p^{-n}(f^n)^*(\omega^p)$ is true under slightly more relaxed conditions (1) through (3) as follows. Notice that $p$ is not necessarily equal to $\pmain$ in the following.
	\begin{enumerate}
		\item $d_{p-1}<d_p$,
		\item $d_p$ is a simple eigenvalue of $f^*$ on $H^{p, p}(X, \R)$, and
		\item the other (real or complex) eigenvalues of this operator on $H^{p, p}(X, \R)$ have modulus strictly smaller than $d_p$.
	\end{enumerate}
	The proofs of lemmas and propositions in \cite{DNV} work in the same way. So, we have Theorem \ref{thm:almost_simple}.
\end{remark}

\bibliographystyle{plain}
\bibliography{sn-bibliography_1}

\end{document}